\documentclass[oneside,reqno,english]{amsart}
\usepackage{lmodern}
\usepackage[T1]{fontenc}
\usepackage[latin9]{inputenc}
\usepackage{babel}
\usepackage{varioref}
\usepackage{float}
\usepackage{amstext}
\usepackage{amsthm}
\usepackage{amssymb}
\usepackage{graphicx}
\PassOptionsToPackage{normalem}{ulem}
\usepackage{ulem}
\usepackage[unicode=true,pdfusetitle,
 bookmarks=true,bookmarksnumbered=false,bookmarksopen=false,
 breaklinks=false,pdfborder={0 0 1},backref=false,colorlinks=false]
 {hyperref}

\makeatletter

\floatstyle{ruled}
\newfloat{algorithm}{tbp}{loa}
\providecommand{\algorithmname}{Algorithm}
\floatname{algorithm}{\protect\algorithmname}

\numberwithin{equation}{section}
\numberwithin{figure}{section}
\theoremstyle{plain}
\newtheorem{thm}{\protect\theoremname}[section]
\theoremstyle{plain}
\newtheorem{lyxalgorithm}[thm]{\protect\algorithmname}
\theoremstyle{plain}
\newtheorem{assumption}[thm]{\protect\assumptionname}
\theoremstyle{plain}
\newtheorem{prop}[thm]{\protect\propositionname}
\theoremstyle{plain}
\newtheorem{lem}[thm]{\protect\lemmaname}
\theoremstyle{definition}
\newtheorem{example}[thm]{\protect\examplename}
\theoremstyle{definition}
\newtheorem{defn}[thm]{\protect\definitionname}
\theoremstyle{remark}
\newtheorem{rem}[thm]{\protect\remarkname}

\newcommand{\dom}{\mbox{\rm dom}}

\makeatother

\providecommand{\algorithmname}{Algorithm}
\providecommand{\assumptionname}{Assumption}
\providecommand{\definitionname}{Definition}
\providecommand{\examplename}{Example}
\providecommand{\lemmaname}{Lemma}
\providecommand{\propositionname}{Proposition}
\providecommand{\remarkname}{Remark}
\providecommand{\theoremname}{Theorem}

\begin{document}
\title[Dykstra's algorithm: Level sets]{Convergence rate of distributed Dykstra's algorithm with sets defined as level sets of convex functions} 

\subjclass[2010]{68W15, 65K05, 90C25, 90C30}
\begin{abstract}
We investigate the convergence rate of the distributed Dykstra's algorithm
when some of the sets are defined as the level sets of convex functions.
We carry out numerical experiments to compare with the theoretical
results obtained.
\end{abstract}

\author{C.H. Jeffrey Pang}

\thanks{C.H.J. Pang acknowledges grant R-146-000-265-114 from the Faculty
of Science, National University of Singapore. }

\curraddr{Department of Mathematics\\ 
National University of Singapore\\ 
Block S17 08-11\\ 
10 Lower Kent Ridge Road\\ 
Singapore 119076 }

\email{matpchj@nus.edu.sg}

\date{\today{}}

\keywords{Distributed optimization, level sets, Dykstra's algorithm}

\maketitle
\tableofcontents{}

\section{Introduction}

Let $X$ be a finite dimensional Hilbert space. For a finite set $V$,
consider the problem 
\begin{equation}
\begin{array}{c}
\underset{x\in X}{\min}\underset{i\in V}{\overset{\phantom{i\in V}}{\sum}}\delta_{C_{i}}(x)+\frac{1}{2}\|x-\bar{x}\|^{2},\end{array}\label{eq:orig-primal}
\end{equation}
where $\delta_{C_{i}}(\cdot)$ is the indicator function of the set
$C_{i}$ defined by 
\begin{equation}
C_{i}:=\{x:g_{i}(x)\leq0\}\label{eq:level-set-form}
\end{equation}
for some closed convex subdifferentiable function $g_{i}:X\to\mathbb{R}$
with full domain. If $C_{i}$ were sets that are easy to project onto
rather than through \eqref{eq:level-set-form}, then Dykstra's algorithm
\cite{Dykstra83} is one way to solve problem \eqref{eq:orig-primal}.
It was recognized in \cite{Han88} that Dykstra's algorithm is block
coordinate ascent on the dual. We prefer to call it Dykstra's algorithm
because the Boyle-Dykstra theorem \cite{BD86} shows the convergence
to the primal minimizer even if a dual maximizer is absent. (In \cite{Han88}
and most other papers on block coordinate methods, a dual maximizer
is assumed to exist, with a constraint qualification or otherwise.)
The proof in \cite{BD86} is rewritten in the form of mathematical
programming in \cite{Gaffke_Mathar}. 

Solving \eqref{eq:orig-primal} directly may not be easy to do if
only function values and a subgradient of $g(\cdot)$ is easy to obtain
in each iteration. As was discussed in \cite{Combettes_SICON_2000,BCRZ03},
an iterative method to find the minimizer of \eqref{eq:orig-primal}
is to project onto outer approximates 
\begin{equation}
\{x:g_{i}(\tilde{x})+\langle\tilde{s},x-\tilde{x}\rangle\leq0\}\label{eq:subgrad_approx}
\end{equation}
 of $C_{i}$, where $\tilde{x}$ is some point in $X$. Halfspaces
like \eqref{eq:subgrad_approx} are easier to project onto than $C_{i}$
itself. The method proposed in \cite{Combettes_SICON_2000} shares
more similarity with Haugazeau's algorithm \cite{Haugazeau68}. 

In \cite{BCRZ03}, the authors extend Dykstra's algorithm while projecting
onto supersets of $C_{i}$ (not necessarily of the form \eqref{eq:subgrad_approx}),
showing the convergence to a primal minimizer under a constraint qualification.

In a series of papers \cite{Pang_Dyk_spl,Pang_Dist_Dyk,Pang_sub_Dyk},
we showed that extending Dykstra's algorithm leads to a distributed
optimization algorithm for problems of the form 
\[
\begin{array}{c}
\underset{x\in X_{i}}{\min}\underset{i\in V}{\sum}[f_{i}(x)+\frac{1}{2}\|x-\bar{x}_{i}\|^{2}],\end{array}
\]
where $X_{i}$ are finite dimensional Hilbert space, and $f_{i}:X_{i}\to\mathbb{R}\cup\{\infty\}$
are closed convex functions that are either proximable, or subdifferentiable
with full domain. The algorithm, which we call the distributed Dykstra's
algorithm, has many favorable properties. Such properties include
being distributed, asynchronous, decentralized (similar to the special
case of the averaged consensus problem), and having deterministic
convergence rates. Other properties include being applicable for time-varying
graphs, allow partial communication of data (so that computations
are not limited by communication speeds), and having convergence rates
compatible with various first order algorithms for various special
cases. 

\subsection{Contributions of this paper}

It appears that the convergence rates of Dykstra's algorithm for \eqref{eq:orig-primal}
has not been studied. In this paper, we study the convergence rate
of the distributed Dykstra's algorithm when the outer approximates
of the form \eqref{eq:subgrad_approx} are used. We show that our
algorithm has $O(1/k)$ convergence (for the dual objective value)
for the case when $|V|=1$ in \eqref{eq:orig-primal}, and $O(1/k^{1/3})$
convergence for the distributed Dykstra's algorithm. We also perform
numerical experiments to compare the theoretical rates obtained.

\subsection{Notation}

Throughout this paper, we assume that the Hilbert spaces are finite
dimensional. We denote $P_{C}(x)$ to be the projection of $x$ onto
the set $C$. Other notations are standard in convex analysis. 

\section{The case of one set}

In this section, we work on the case \eqref{eq:orig-primal} when
$|V|=1$.  The primal problem and its (Fenchel) dual are 
\begin{equation}
\begin{array}{c}
(P)\,\,\underset{x\in X}{\overset{\phantom{x\in\mathbb{R}^{m}}}{\min}}\frac{1}{2}\|x-\bar{x}\|^{2}+f(x),\,\,\text{ and }\,\,(D)\,\,\underset{z\in X}{\max}\frac{1}{2}\|\bar{x}\|^{2}-\frac{1}{2}\|z-\bar{x}\|^{2}-f^{*}(z),\end{array}\label{eq:p-d-1-case}
\end{equation}
where $f(\cdot)=\delta_{\{x:g(x)\leq0\}}(\cdot)$ and $g:X\to\mathbb{R}$
is a subdifferentiable function with full domain. Strong duality holds
for \eqref{eq:p-d-1-case}. We now look at the basic algorithm in
Algorithm \vref{eq:p-d-1-case}. Note the similarities of the Algorithm
\ref{alg:basic} to Haugazeau's algorithm \cite{Haugazeau68}; See
also \cite{BauschkeCombettes11}. We make the following assumption
on $g(\cdot)$.

\begin{algorithm}[h]
\begin{lyxalgorithm}
\label{alg:basic}This algorithm finds iterates $\{x_{k}\}_{k}$ that
converges to the solution of \eqref{eq:p-d-1-case}.

Set $x_{0}=\bar{x}$.

Set $H_{0}=\{x:c_{0}^{T}x\leq b\}$ so that $C\subset H_{0}$. (Note:
$c_{0}$ and $b$ can be chosen to be $0$)

For $k=0,1,\dots$

$\quad$Find $\tilde{C}_{k}$ such that $C\subset\tilde{C}_{k}$ and
$x_{k}\notin\tilde{C}_{k}$. A typical choice is 
\begin{equation}
\tilde{C}_{k}=\{x:g(x_{k})+\langle x-x_{k},s_{k}\rangle\leq0\}\text{ for some }s_{k}\in\partial g(x_{k}).\label{eq:tilde_C_choice}
\end{equation}

$\quad$Let $x_{k+1}=P_{H_{k}\cap\tilde{C}_{k}}(\bar{x})$. 

$\quad$Let $H_{k+1}$ be the halfspace such that $x_{k+1}=P_{H_{k+1}}(\bar{x})$. 

End For
\end{lyxalgorithm}

\end{algorithm}

\begin{assumption}
\label{assu:g-negative}Suppose that $g:X\to\mathbb{R}$ is a subdifferentiable
function with full domain and $\min_{x\in X}g(x)<0$. 
\end{assumption}

If $\min_{x}g(x)>0$, then the problem is infeasible. If $\min_{x}g(x)=0$,
then note that a slight perturbation of $g(\cdot)$ would render the
problem infeasible. See \cite[Theorem 9.41(b)]{RW98} for more connections
to stability. So this assumption ensures the stability of the problem. 
\begin{prop}
\label{prop:subgrad-bdd}Suppose Assumption \ref{assu:g-negative}
holds and let $R$ be a bounded set. Then there is some $c>0$ such
that if $x\in R$, $g(x)\geq0$ and $s\in\partial g(x)$, then $\|s\|>c$. 
\end{prop}

\begin{proof}
Seeking a contradiction, suppose $(x_{i},s_{i})$ satisfies the conditions
and $\lim_{i\to\infty}\|s_{i}\|=0$, $g(x_{i})\geq0$. Let $\hat{x}$
be $\lim_{i\to\infty}x_{i}$. By the outer semicontinuity of the subgradient
mapping, $0\in\partial g(\hat{x})$ and $g(\hat{x})\geq0$, which
contradicts Assumption \ref{assu:g-negative}.
\end{proof}
\begin{lem}
\label{lem:3_pts}Let $X$ be a finite dimensional Hilbert space,
and let $g:X\to\mathbb{R}$ be a subdifferentiable function with full
domain satisfying Assumption \ref{assu:g-negative}. Let $C:=\{x:g(x)\leq0\}$,
and let $R$ be a bounded set. Let $\bar{x}\in R$. Let $\hat{x}=P_{C}(\bar{x})$,
let $H$ be a halfspace such that $C\subset H$, and let $x=P_{H}(x)$.
Then the following hold:
\begin{enumerate}
\item $\bar{x}-\hat{x}\in N_{C}(\hat{x})$.
\item There is some constant $\gamma>0$ such that if $g(x)\geq0$ and $s\in\partial g(x)$,
then $\|s\|\geq\gamma$. 
\item Let $\hat{H}$ be the halfspace $\{x:\langle\bar{x}-\hat{x},x-\hat{x}\rangle\geq0\}$.
Then for the constant $\gamma>0$ in (2), $g(x)\geq\gamma d(x,\hat{H})$. 
\item Let $\tilde{R}$ be a compact set. Let $\tilde{\gamma}:=\sup\{\|\tilde{s}\|:\tilde{s}\in\partial g(\tilde{x}),\tilde{x}\in\tilde{R}\}$,
which is finite from the fact that $\dom(g)=X$. For $\tilde{x}$
such that $g(\tilde{x})>0$, let $\tilde{H}$ be the halfspace $\tilde{H}:=\{x:g(\tilde{x})+\langle x-\tilde{x},\tilde{s}\rangle\leq0\}$.
Then $d(\tilde{x},\tilde{H})=g(\tilde{x})/\|\tilde{s}\|$. Moreover,
if $\tilde{x}\in\tilde{R}$, we have $d(\tilde{x},\tilde{H})\geq g(\tilde{x})/\tilde{\gamma}>0$. 
\end{enumerate}
\end{lem}

\begin{proof}
Property (1) is obvious. We now prove (2) by contradiction. Since
$\bar{x}$ lies in the bounded set $R$, $P_{C}(\bar{x})$ also lies
in a bounded set. Since $x$ lies in the ball with center $\bar{x}$
and radius $\|\bar{x}-\hat{x}\|$, $x$ lies in a bounded set as well.
Apply Proposition \ref{prop:subgrad-bdd}.

Next, we prove (3). By the optimality conditions on $\hat{x}$, there
is some subgradient $\hat{s}\in\partial g(\hat{x})$ that is a positive
multiple of $\bar{x}-\hat{x}$. Then 
\[
g(x)\geq g(\hat{x})+\langle\hat{s},x-\hat{x}\rangle\overset{g(\hat{x})=0}{=}\|\hat{s}\|d(x,\hat{H})\overset{(2)}{\geq}\gamma d(x,\hat{H}).
\]
Lastly, (4) is elementary.
\end{proof}
\begin{lem}
\label{lem:d_sqr_increase} Let $X$ be a finite dimensional Hilbert
space, $\bar{x}\in X$, $H_{1}$ be a halfspace, and $x_{1}=P_{H_{1}}(\bar{x})$.
Let $H_{2}$ be a halfspace, and let $d=d(x_{1},H_{2})$. Let $x_{2}=P_{H_{1}\cap H_{2}}(\bar{x})$.
Then $\|\bar{x}-x_{2}\|^{2}\geq\|\bar{x}-x_{1}\|^{2}+d^{2}$. 
\end{lem}

\begin{proof}
Since $x_{2}\in H_{1}$ and $x_{1}=P_{H_{1}}(\bar{x})$, we have
$\langle\bar{x}-x_{1},x_{2}-x_{1}\rangle\leq0$. Also, since $x_{2}\in H_{2}$,
we have $\|x_{1}-x_{2}\|\geq d$. Hence 
\[
\|\bar{x}-x_{2}\|^{2}=\|\bar{x}-x_{1}\|^{2}+\|x_{1}-x_{2}\|^{2}+2\langle\bar{x}-x_{1},x_{1}-x_{2}\rangle\geq\|\bar{x}-x_{1}\|^{2}+d^{2}.
\]
\end{proof}
The following result gives convergence rates for sequences defined
by recurrences. 
\begin{lem}
\label{lem:Beck-recur}Let $\{a_{k}\}_{k=1}^{\infty}$ be a nonnegative
sequence. 
\begin{enumerate}
\item Suppose $\{a_{k}\}_{k=1}^{\infty}$ has the recurrence $a_{k}\geq a_{k+1}+\theta a_{k}^{2}$
for some $\theta>0$. Then $\{a_{k}\}_{k}$ has a $O(1/k)$ rate of
convergence. 
\item Suppose $\{a_{k}\}_{k=1}^{\infty}$ has the recurrence $a_{k}\geq a_{k+1}+\theta a_{k+1}^{4}$
for some $\theta>0$. Then $a_{k}\leq\left(\frac{1}{a_{1}^{3}}+(k-1)3\theta\left(3\theta a_{1}^{3}+1\right)^{-1}\right)^{-1/3}$
for all $k\geq1$, which means that $\{a_{k}\}_{k}$ has a $O(1/k^{1/3})$
rate of convergence. 
\end{enumerate}
\end{lem}

\begin{proof}
Part (1) was addressed in \cite[Lemma 6.2]{Beck_Tetruashvili_2013}
and \cite[Lemma 3.8]{Beck_alt_min_SIOPT_2015}. Part (2) was addressed
in \cite{Pang_rate_D_Dyk}. 
\end{proof}
We now turn to the problem we try to prove. Let $\bar{d}$ be the
distance $d(\bar{x},C)$, $\hat{x}:=P_{C}(\bar{x})$ so that $\bar{d}=\|\hat{x}-\bar{x}\|$,
and $d_{k}:=\|\bar{x}-x_{k}\|$. The objective value of \eqref{eq:p-d-1-case}
is $\frac{1}{2}\bar{d}^{2}$. Making use of Lemma \ref{lem:3_pts}(3),
we observe that 
\begin{equation}
g(x_{k})\overset{\scriptsize{\text{Lem \ref{lem:3_pts}(3)}}}{\geq}\gamma d(x_{k},\hat{H})\geq\gamma(\|\hat{x}-\bar{x}\|-\|x_{k}-\bar{x}\|)=\gamma(\bar{d}-d_{k}).\label{eq:g_x_k_ineq}
\end{equation}
Moreover, by Lemma \ref{lem:3_pts}(4), $d(x_{k},\tilde{C}_{k})\geq g(x_{k})/\tilde{\gamma}$.
We have 
\begin{equation}
\begin{array}{c}
d_{k+1}^{2}\overset{\scriptsize{\text{Lem \ref{lem:d_sqr_increase}}}}{\geq}d_{k}^{2}+d(x_{k},\tilde{C}_{k})^{2}\overset{\scriptsize{\text{Lem \ref{lem:3_pts}(4), \eqref{eq:g_x_k_ineq}}}}{\geq}d_{k}^{2}+[\frac{\gamma}{\tilde{\gamma}}]^{2}(\bar{d}-d_{k})^{2}.\end{array}\label{eq:before_v}
\end{equation}
Let $v_{k}=\bar{d}^{2}-d_{k}^{2}$. Note that the objective value
of $\min_{x\in X}\frac{1}{2}\|\bar{x}-x\|^{2}+\delta_{H_{k}}(x)$
is $\frac{1}{2}d_{k}^{2}$, while objective value of \eqref{eq:orig-primal}
with $|V|=1$ is $\frac{1}{2}\bar{d}^{2}$. In other words, $v_{k}$
is twice the gap between the actual and estimated objective values.
We have 
\begin{equation}
\begin{array}{c}
v_{k+1}\overset{\eqref{eq:before_v}}{\leq}v_{k}-\frac{\gamma^{2}}{\tilde{\gamma}^{2}}(\bar{d}-d_{k})^{2}=v_{k}-\frac{\gamma^{2}}{\tilde{\gamma}^{2}}\frac{v_{k}^{2}}{(\bar{d}+d_{k})^{2}}\leq v_{k}-\frac{\gamma^{2}}{\tilde{\gamma}^{2}}\frac{v_{k}^{2}}{2\bar{d}}.\end{array}\label{eq:1-k-recur}
\end{equation}
By Lemma \ref{lem:Beck-recur}(1), $v_{k}$ converges to zero at a
$O(1/k)$ rate. 

\section{Preliminaries from \cite{Pang_rate_D_Dyk}}

In this section, we list down the preliminaries and description of
the distributed Dykstra's algorithm studied in \cite{Pang_Dist_Dyk,Pang_sub_Dyk}.
We do not claim originality in this section, and we recall some results
useful for the subsequent proofs. 

Let $V$ and $\bar{E}$ be finite sets. Define the set $\mathbf{X}:=X_{1}\times\cdots\times X_{|V|}$,
where each $X_{i}$ is a finite dimensional Hilbert space. For each
$i\in V$, let $f_{i}:X_{i}\to\mathbb{R}\cup\{\infty\}$ be a closed
convex function, and let $\mathbf{f}_{i}:\mathbf{X}\to\mathbb{R}\cup\{\infty\}$
be defined by 
\begin{equation}
\mathbf{f}_{i}(\mathbf{x})=f_{i}([\mathbf{x}]_{i}).\label{eq:f-component}
\end{equation}
Let $\delta_{C}(\cdot)$ be the indicator function for a closed convex
set $C$. For each $\alpha\in\bar{E}$, let $H_{\alpha}\subset\mathbf{X}$
be a linear subspace, and define $\mathbf{f}_{\alpha}:\mathbf{X}\to\mathbb{R}$
by $\mathbf{f}_{\alpha}(\mathbf{x})=\delta_{H_{\alpha}}(\mathbf{x})$.
The setting for the distributed Dykstra's algorithm is 
\begin{equation}
\begin{array}{c}
\underset{\mathbf{x}\in\mathbf{X}}{\min}\frac{1}{2}\|\mathbf{x}-\bar{\mathbf{x}}\|^{2}+\underset{i\in V}{\sum}\mathbf{f}_{i}(\mathbf{x})+\underset{\alpha\in\bar{E}}{\overset{\phantom{\alpha\in\bar{E}}}{\sum}}\delta_{H_{\alpha}}(\mathbf{x}).\end{array}\label{eq:general-framework}
\end{equation}
Note that the last two sums in \eqref{eq:general-framework} can be
written as $\sum_{\alpha\in V\cup\bar{E}}\mathbf{f}_{\alpha}(\mathbf{x})$.
Typically, the hyperplanes $\{H_{\alpha}\}_{\alpha\in\bar{E}}$ are
overdetermined (see Definition \ref{def:connects} later). Partition
the set $V$ as the disjoint union \emph{$V=\cup_{i=1}^{5}V_{5}$
}so that 
\begin{itemize}
\item $f_{i}(\cdot)$ are proximable functions for all $i\in V_{1}$.
\item $f_{i}(\cdot)$ are indicator functions of closed convex sets for
all $i\in V_{2}$. 
\item $f_{i}(\cdot)$ are proximable functions such that $\dom(f_{i})=X_{i}$
for all $i\in V_{3}$.
\item \emph{$f_{i}(\cdot)$ }are subdifferentiable functions (i.e., a subgradient
is easy to obtain) such that $\dom(f_{i})=X_{i}$ for all $i\in V_{4}$
and $f_{i}(\cdot)$ have Lipschitz constant $L$. 
\item $f_{i}(\cdot)$ are indicator functions $\delta_{C_{i}}(\cdot)$,
where $C_{i}:=\{x:g_{i}(x)\leq0\}$, and $g_{i}:X_{i}\to\mathbb{R}$
is a closed convex subdifferentiable function with $\dom(g_{i})=X_{i}$
for all $i\in V_{5}$.
\end{itemize}
The (Fenchel) dual of \eqref{eq:general-framework} can be found to
be 
\begin{equation}
\max_{\mathbf{z}_{\alpha}\in\mathbf{X}:\alpha\in V\cup\bar{E}}F(\{\mathbf{z}_{\alpha}\}_{\alpha\in\bar{E}\cup V}),\label{eq:general-dual}
\end{equation}
where
\begin{equation}
\begin{array}{c}
F(\{\mathbf{z}_{\alpha}\}_{\alpha\in\bar{E}\cup V}):=-\frac{1}{2}\bigg\|\bar{\mathbf{x}}-\underset{\alpha\in\bar{E}\cup V}{\sum}\mathbf{z}_{\alpha}\bigg\|^{2}+\frac{1}{2}\|\bar{\mathbf{x}}\|^{2}-\underset{\alpha\in\bar{E}\cup V}{\sum}\mathbf{f}_{\alpha}^{*}(\mathbf{z}_{\alpha}).\end{array}\label{eq:h-def}
\end{equation}
We now explain that the problem \eqref{eq:general-framework} includes
the general case of the distributed Dykstra's algorithm in \cite{Pang_Dist_Dyk,Pang_sub_Dyk}.
\begin{example}
\label{exa:distrib-dyk} \cite{Pang_Dist_Dyk,Pang_sub_Dyk}(Distributed
Dykstra's algorithm is a special case of \eqref{eq:general-framework})
Let $G=(V,E)$ be an undirected connected graph. Suppose each $X_{i}=\mathbb{R}^{m}$
for all $i\in V$, and let $\bar{E}:=E\times\{1,\dots,m\}$. For each
$\mathbf{x}\in\mathbf{X}=[\mathbb{R}^{m}]^{|V|}$, we let $[\mathbf{x}]_{i}\in\mathbb{R}^{m}$
be the $i$-th component, and we let $[[\mathbf{x}]_{i}]_{k}$ be
the $k$-th component of $[\mathbf{x}]_{i}$. For each $((i,j),k)\in\bar{E}$,
let the linear subspace $H_{((i,j),k)}\subset\mathbf{X}$ of codimension
1 be defined to be 
\begin{equation}
H_{((i,j),k)}:=\{\mathbf{x}\in\mathbf{X}:[[\mathbf{x}]_{i}]_{k}=[[\mathbf{x}]_{j}]_{k}\}.\label{eq:H-alpha-subspaces}
\end{equation}
Then the problem \eqref{eq:general-framework} is equivalent to 
\begin{equation}
\begin{array}{c}
\underset{x\in\mathbb{R}^{m}}{\min}\underset{i\in V}{\overset{\phantom{i\in V}}{\sum}}[\frac{1}{2}\|x-[\bar{\mathbf{x}}]_{i}\|^{2}+f_{i}(x)].\end{array}\label{eq:distrib-dyk-primal-pblm}
\end{equation}
\end{example}

For all $n\geq1$ and $w\in\{1,\dots,\bar{w}\}$, define $\mathbf{f}_{\alpha,n,w}:\mathbf{X}\to\mathbb{R}\cup\{\infty\}$
by \begin{subequations}\label{eq_m:h-a-n-w} 
\begin{eqnarray}
\mathbf{f}_{\alpha,n,w}(\cdot) & = & \mathbf{f}_{\alpha}(\cdot)\mbox{ for all }\alpha\in[\bar{E}\cup V]\backslash[V_{4}\cup V_{5}]\label{eq:h-a-n-w-eq-h-a}\\
\mbox{ and }\mathbf{f}_{\alpha,n,w}(\cdot) & \leq & \mathbf{f}_{\alpha}(\cdot)\mbox{ for all }\alpha\in V_{4}\cup V_{5}.\label{eq:h-a-n-w-lesser}
\end{eqnarray}
\end{subequations}For $i\in V_{4}$, the $f_{i,n,w}(\cdot)\leq f_{i}(\cdot)$
are obtained by taking affine minorants of $f_{i}(\cdot)$, as discussed
in \cite{Pang_sub_Dyk,Pang_rate_D_Dyk}, and then $\mathbf{f}_{i,n,w}(\mathbf{x})=f_{i,n,w}([\mathbf{x}]_{i})$.
For $i\in V_{5}$, define $g_{i,n,w}:X_{i}\to\mathbb{R}$ by taking
affine minorants of $g_{i}(\cdot)$ so that $g_{i,n,w}(\cdot)\leq g_{i}(\cdot)$.
Define $f_{i,n,w}:X_{i}\to\mathbb{R}\cup\{\infty\}$ by $f_{i,n,w}(\cdot):=\delta_{\{x':g_{i,n,w}(x')\leq0\}}(\cdot)$
so that 
\[
f_{i,n,w}(x)=\delta_{\{x':g_{i,n,w}(x')\leq0\}}(x)\leq\delta_{\{x':g_{i}(x)\leq0\}}(x)=f_{i}(x)\text{ for all }i\in V_{5},x\in X_{i},
\]
which leads to \eqref{eq:h-a-n-w-lesser}. Define the function $F^{n,w}:\mathbf{X}^{|V\cup\bar{E}|}\to\mathbb{R}\cup\{\infty\}$
to be 
\begin{equation}
\begin{array}{c}
F^{n,w}(\{\mathbf{z}_{\alpha}\}_{\alpha\in\bar{E}\cup V}):=-\frac{1}{2}\bigg\|\bar{\mathbf{x}}-\underset{\alpha\in\bar{E}\cup V}{\sum}\mathbf{z}_{\alpha}\bigg\|^{2}+\frac{1}{2}\|\bar{\mathbf{x}}\|^{2}-\underset{\alpha\in\bar{E}\cup V}{\sum}\mathbf{f}_{\alpha,n,w}^{*}(\mathbf{z}_{\alpha}).\end{array}\label{eq:h-def-1}
\end{equation}
By \eqref{eq_m:h-a-n-w}, $F^{n,w}(\cdot)$ is a lower approximate
of $F(\cdot)$ for the maximization problem \eqref{eq:general-dual}.
Based on our original motivation in Example \ref{exa:distrib-dyk}
from \cite{Pang_Dist_Dyk,Pang_sub_Dyk}, we make the following definition. 
\begin{defn}
\label{def:connects}Let $D:=\cap_{\alpha\in\bar{E}}H_{\alpha}$.
We say that a subset $E'\subset\bar{E}$ \emph{connects} $V$ if 
\begin{equation}
\cap_{\alpha\in E'}H_{\alpha}=D.\label{eq:intersect-H}
\end{equation}
\end{defn}

Since $H_{\alpha}$ were assumed to be linear subspaces, it is clear
that condition \eqref{eq:intersect-H} on $E'$ is equivalent to 
\begin{equation}
\begin{array}{c}
\underset{\alpha\in E'}{\overset{\phantom{\alpha\in E'}}{\sum}}H_{\alpha}^{\perp}=D^{\perp}.\end{array}\label{eq:D-perp}
\end{equation}
To provide more intuition, note that the set $D$ defined through
\eqref{eq:intersect-H}  has the simplifications\begin{subequations}
\begin{eqnarray}
 &  & \begin{array}{c}
D^{\phantom{\perp}}=\{\mathbf{x}\in[\mathbb{R}^{m}]^{|V|}:\mathbf{x}=(x,x,\dots,x)\mbox{ for some }x\in\mathbb{R}^{m}\}\end{array}\label{eq:D-formula-a}\\
 & \mbox{ and } & \begin{array}{c}
D^{\perp}=\Big\{\mathbf{x}\in[\mathbb{R}^{m}]^{|V|}:\underset{i\in V}{\overset{\phantom{i\in V}}{\sum}}[\mathbf{x}]_{i}=0\Big\},\end{array}\label{eq:D-formula-b}
\end{eqnarray}
\end{subequations}which are consistent with the usual product space
formulation. 

The following simple result was needed in \cite{Pang_Dist_Dyk} in
order to show that the distributed Dykstra's algorithm works for time
varying graphs, but it is not needed here. Nevertheless, this result
explains line 5 of Algorithm \ref{alg:Ext-Dyk}. The proof is exactly
the same as in \cite{Pang_rate_D_Dyk}.
\begin{lem}
\cite{Pang_rate_D_Dyk} \label{lem:express-v}There is a constant
$c_{reg}>0$ such that for any $\mathbf{v}\in D^{\perp}$ and any
$E'\subset\bar{E}$ such that $E'$ connects $V$, we can write $\mathbf{v}=\sum_{\alpha\in E'}\mathbf{v}_{\alpha}$
so that\textbf{ $\mathbf{v}_{\alpha}\in H_{\alpha}^{\perp}$ }and
$\|\mathbf{v}_{\alpha}\|\leq c_{reg}\|\mathbf{v}\|$ for all $\alpha\in E'$. 
\end{lem}

To simplify calculations, we let the vectors $\mathbf{v}_{A}$, $\mathbf{v}_{H}$
and $\mathbf{x}$ in $\mathbf{X}$ be denoted by\begin{subequations}\label{eq_m:all_acronyms}
\begin{eqnarray}
\mathbf{v}_{H} & = & \begin{array}{c}
\underset{\alpha\in\bar{E}}{\overset{\phantom{i\in V}}{\sum}}\mathbf{z}_{\alpha}\end{array}\label{eq:v-H-def}\\
\mathbf{v}_{A} & = & \begin{array}{c}
\mathbf{v}_{H}+\underset{i\in V}{\overset{\phantom{i\in V}}{\sum}}\mathbf{z}_{i}\end{array}\label{eq:from-10}\\
\mathbf{x} & = & \begin{array}{c}
\bar{\mathbf{x}}-\mathbf{v}_{A}.\end{array}\label{eq:x-from-v-A}
\end{eqnarray}
\end{subequations}

We now state Algorithm \vref{alg:Ext-Dyk}. Algorithm \ref{alg:Ext-Dyk}
calls on Algorithm \vref{alg:subdiff-subalg} as a subalgorithm. 

\begin{algorithm}[!h]
\begin{lyxalgorithm}
\label{alg:Ext-Dyk}(Distributed Dykstra's algorithm) Consider the
problem \eqref{eq:general-framework} along with the associated dual
problem \eqref{eq:general-dual}.

Let $\bar{w}$ be a positive integer. Let $c_{reg}>0$ satisfy Lemma
\ref{lem:express-v}. For each $\alpha\in[\bar{E}\cup V]\backslash[V_{4}\cup V_{5}]$,
$n\geq1$ and $w\in\{1,\dots,\bar{w}\}$, let $\mathbf{f}_{\alpha,n,w}:\mathbf{X}\to\mathbb{R}$
be as defined in \eqref{eq_m:h-a-n-w}. Our distributed Dykstra's
algorithm is as follows:

01$\quad$Let 

\begin{itemize}
\item $\mathbf{z}_{i}^{1,0}\in\mathbf{X}$ be a starting dual vector for
$\mathbf{f}_{i}(\cdot)$ for each $i\in V$ so that $[\mathbf{z}_{i}^{1,0}]_{j}=0\in X_{j}$
for all $j\in V\backslash\{i\}$. 
\item $\mathbf{v}_{H}^{1,0}\in D^{\perp}$ be a starting dual vector.

\begin{itemize}
\item Note: $\{\mathbf{z}_{\alpha}^{n,0}\}_{\alpha\in\bar{E}}$ is defined
through $\mathbf{v}_{H}^{n,0}$ in \eqref{eq_m:resetted-z-i-j}.
\end{itemize}
\item Let $\mathbf{x}^{1,0}$ be $\mathbf{x}^{1,0}=\bar{\mathbf{x}}-\mathbf{v}_{H}^{1,0}-\sum_{i\in V}\mathbf{z}_{i}^{1,0}$.
\end{itemize}
02$\quad$For each $i\in V_{4}$, let $\mathbf{f}_{i,1,0}:\mathbf{X}\to\mathbb{R}$
be a function such that $\mathbf{f}_{i,1,0}(\cdot)\leq\mathbf{f}_{i}(\cdot)$

$\phantom{\text{02}}\quad$For each $i\in V_{5}$, let $g_{i,1,0}:X_{i}\to\mathbb{R}$
be a function such that $g_{i,1,0}(\cdot)\leq g(\cdot)$. 

03$\quad$For $n=1,2,\dots$

04$\quad$$\quad$\textup{Let $\bar{E}_{n}\subset\bar{E}$ be such
that $\bar{E}_{n}$ connects $V$ in the sense of Definition \ref{def:connects}.}

05$\quad$$\quad$\textup{Define $\{\mathbf{z}_{\alpha}^{n,0}\}_{\alpha\in\bar{E}}$
so that:\begin{subequations}\label{eq_m:resetted-z-i-j} 
\begin{eqnarray}
\mathbf{z}_{\alpha}^{n,0} & = & 0\mbox{ for all }\alpha\notin\bar{E}_{n}\label{eq:reset-z-i-j-1}\\
\mathbf{z}_{\alpha}^{n,0} & \in & H_{\alpha}^{\perp}\mbox{ for all }\alpha\in\bar{E}\label{eq:reset-z-i-j-2}\\
\|\mathbf{z}_{\alpha}^{n,0}\| & \leq & c_{reg}\|\mathbf{v}_{H}^{n,0}\|\mbox{ for all }\alpha\in\bar{E}\label{eq:reset-z-i-j-3}\\
\mbox{ and }\sum_{\alpha\in\bar{E}}\mathbf{z}_{\alpha}^{n,0} & = & \mathbf{v}_{H}^{n,0}.\label{eq:reset-z-i-j-4}
\end{eqnarray}
\end{subequations}}

$\quad$$\quad$\textup{(This is possible by Lemma \ref{lem:express-v}.) }

06$\quad$$\quad$For $w=1,2,\dots,\bar{w}$

07$\quad$$\quad$$\quad$Choose a set $S_{n,w}\subset\bar{E}_{n}\cup V$
such that $S_{n,w}\neq\emptyset$. 

08$\quad$$\quad$$\quad$If $S_{n,w}\subset V_{4}\cup V_{5}$, then 

09$\quad$$\quad$$\quad$$\quad$Apply Algorithm \ref{alg:subdiff-subalg}.

10$\quad$$\quad$$\quad$else

11$\quad$$\quad$$\quad$$\quad$Set $\mathbf{f}_{i,n,w}(\cdot):=\mathbf{f}_{i,n,w-1}(\cdot)$
for all $i\in V_{4}$.

12$\quad$$\quad$$\quad$$\quad$Define $\{\mathbf{z}_{\alpha}^{n,w}\}_{\alpha\in S_{n,w}}$
by 
\begin{equation}
\begin{array}{c}
\{\mathbf{z}_{\alpha}^{n,w}\}_{\alpha\in S_{n,w}}=\underset{\mathbf{z}_{\alpha},\alpha\in S_{n,w}}{\arg\min}\frac{1}{2}\Big\|\bar{\mathbf{x}}-\underset{\alpha\notin S_{n,w}}{\overset{\phantom{\alpha\notin S_{n,w}}}{\sum}}\mathbf{z}_{\alpha}^{n,w-1}-\underset{\alpha\in S_{n,w}}{\sum}\mathbf{z}_{\alpha}\Big\|^{2}+\underset{\alpha\in S_{n,w}}{\overset{\phantom{\alpha\in S_{n,w}}}{\sum}}\mathbf{f}_{\alpha,n,w}^{*}(\mathbf{z}_{\alpha}).\end{array}\label{eq:Dykstra-min-subpblm}
\end{equation}

13$\quad$$\quad$$\quad$end if 

14$\quad$$\quad$$\quad$Set $\mathbf{z}_{\alpha}^{n,w}:=\mathbf{z}_{\alpha}^{n,w-1}$
for all $\alpha\notin S_{n,w}$.

15$\quad$$\quad$End For 

16$\quad$$\quad$Let $\mathbf{z}_{i}^{n+1,0}=\mathbf{z}_{i}^{n,\bar{w}}$
for all $i\in V$ and $\mathbf{v}_{H}^{n+1,0}=\mathbf{v}_{H}^{n,\bar{w}}=\sum_{\alpha\in\bar{E}}\mathbf{z}_{\alpha}^{n,\bar{w}}$.

17$\quad$$\quad$Let $\mathbf{f}_{i,n+1,0}(\cdot)=\mathbf{f}_{i,n,\bar{w}}(\cdot)$
for all $i\in V_{4}$.

$\phantom{\text{17}}\quad$$\quad$Let $g_{i,n+1,0}(\cdot)=g_{i,n,\bar{w}}(\cdot)$
for all $i\in V_{5}$.

18$\quad$End For 
\end{lyxalgorithm}

\end{algorithm}

\begin{algorithm}[!h]
\begin{lyxalgorithm}
\label{alg:subdiff-subalg}(Subalgorithm for outer approximates of
$C_{i}$) This algorithm is run when line 9 of Algorithm \ref{alg:Ext-Dyk}
is reached. Note that to get to this subalgorithm, $S_{n,w}\subset V_{4}\cup V_{5}$.
Suppose Assumption \ref{assu:to-start-subalg} holds.

00 For all $i\in S_{n,w}\cap V_{4}$, use steps in the corresponding
algorithm in \cite{Pang_rate_D_Dyk}.

01 For each $i\in S_{n,w}\cap V_{5}$ 

02 $\quad$For $\tilde{g}_{i,n,w-1}:X_{i}\to\mathbb{R}$ defined by
\begin{equation}
\tilde{g}_{i,n,w-1}(x):=g_{i}([\bar{\mathbf{x}}-\mathbf{v}_{H}^{n,w-1}-\mathbf{z}_{i}^{n,w-1}]_{i})+\langle s,x-[\bar{\mathbf{x}}-\mathbf{v}_{H}^{n,w-1}-\mathbf{z}_{i}^{n,w-1}]_{i}\rangle,\label{eq:linearize-f-i-n-w}
\end{equation}

where $s\in\partial g_{i}([\bar{\mathbf{x}}-\mathbf{v}_{H}^{n,w-1}-\mathbf{z}_{i}^{n,w-1}]_{i})$,
consider 
\begin{equation}
\begin{array}{c}
\underset{x\in X_{i}}{\min}\left[\frac{1}{2}\|[\bar{\mathbf{x}}-\mathbf{v}_{H}^{n,w-1}]_{i}-x\|^{2}\text{ s.t. }g_{i,n,w-1}(x)\leq0\text{ and }\tilde{g}_{i,n,w-1}(x)\leq0\right].\end{array}\label{eq:alg-primal-subpblm}
\end{equation}

03 $\quad$Let the primal of \eqref{eq:alg-primal-subpblm} be $x_{i}^{+}$,
and its dual solution $[\bar{\mathbf{x}}-\mathbf{v}_{H}^{n,w-1}]_{i}-x_{i}^{+}$
be $z_{i}^{+}$. 

04 $\quad$Define $g_{i,n,w}:X_{i}\to\mathbb{R}$ to be the affine
function 
\begin{equation}
g_{i,n,w}(x):=g_{i,n,w-1}(x_{i}^{+})+\langle x-x_{i}^{+},[\bar{\mathbf{x}}-\mathbf{v}_{H}^{n,w-1}]_{i}-x_{i}^{+}\rangle.\label{eq:def-f-i-n-w}
\end{equation}

05 $\quad$In other words, $g_{i,n,w}(\cdot)$ is chosen such that
the 

$\qquad\qquad$primal and dual optimizers to \eqref{eq:alg-primal-subpblm}
coincide with that of 
\begin{equation}
\begin{array}{c}
\underset{x\in X_{i}}{\min}\left[\frac{1}{2}\|[\bar{\mathbf{x}}-\mathbf{v}_{H}^{n,w-1}]_{i}-x\|^{2}\text{ s.t. }g_{i,n,w}(x)\leq0\right].\end{array}\label{eq:finw-design}
\end{equation}

06 $\quad$Define the dual vector $\mathbf{z}_{i}^{n,w}\in\mathbf{X}$
to be 
\begin{equation}
[\mathbf{z}_{i}^{n,w}]_{j}:=\begin{cases}
z_{i}^{+} & \mbox{ if }j=i\\
0 & \mbox{ if }j\neq i.
\end{cases}\label{eq:def-z-i}
\end{equation}

07 End for

08 For all $i\in V_{4}\backslash S_{n,w}$, $\mathbf{f}_{i,n,w}(\cdot)=\mathbf{f}_{i,n,w-1}(\cdot)$.

09 For all $i\in V_{5}\backslash S_{n,w}$, $g_{i,n,w}(\cdot)=g_{i,n,w-1}(\cdot)$.
\end{lyxalgorithm}

\end{algorithm}

\begin{rem}
(Intuition behind Algorithms \ref{alg:Ext-Dyk} and \ref{alg:subdiff-subalg})
We summarize the intuition behind Algorithms \ref{alg:Ext-Dyk} and
\ref{alg:subdiff-subalg}. Dykstra's algorithm is block coordinate
ascent on the dual \eqref{eq:general-dual}, and this is reflected
in lines 7-14 of Algorithm \ref{alg:Ext-Dyk}. That is, find $\mathbf{z}\in\mathbf{X}^{|\bar{E}\cup V|}$
that tries to improve the objective value of \eqref{eq:h-def}. As
explained in \cite{Pang_Dist_Dyk}, one only needs to keep track of
$\mathbf{x}_{i}$ and $[\mathbf{z}_{i}]_{i}$ for all $i\in V$, and
not all the variables. Line 5 corrects $\{\mathbf{z}_{\alpha}\}_{\alpha\in\bar{E}}$
so that the dual objective value remains the same, and this consideration
is needed when we try to prove that the algorithm works for time-varying
graphs. What is new in this paper is the consideration for $i\in V_{5}$.
 When $i\in V_{5}$, we have $f_{i}=\delta_{C_{i}}(\cdot)$, where
$C_{i}=\{x:g_{i}(x)\leq0\}$. Since projection onto $C_{i}$ may not
be easy, we let $f_{i,n,w}(\cdot)=\delta_{\{x:g_{i,n,w}(x)\leq0\}}(\cdot)$,
where $g_{i,n,w}(\cdot)$ is affine. One can see that the projection
onto the halfspace $\{x:g_{i,n,w}(x)\leq0\}$ is easier than that
of $C_{i}$. We have $f_{i,n,w}(\cdot)\leq f_{i}(\cdot)$, which is
also $f_{i,n,w}^{*}(\cdot)\geq f_{i}^{*}(\cdot)$. We shall show that
solving a sequence of problems involving $f_{i,n,w}^{*}(\cdot)$ instead
of $f_{i}^{*}(\cdot)$ would minorize the objective value in \eqref{eq:h-def}
by \eqref{eq:h-def-1}, and allow the dual objective value to converge
to its optimum value, which in turn leads to convergence of the primal
to its minimizer. 
\end{rem}

The following result is essential for showing that the distributed
Dykstra's algorithm is asynchronous, and also in showing that the
problems involving the nodes in $i\in V$ are separable. 
\begin{prop}
\label{prop:sparsity}(Sparsity of $\mathbf{z}_{\alpha}$) We have
$[\mathbf{z}_{i}^{n,w}]_{j}=0$ for all $j\in V\backslash\{i\}$,
$n\geq1$ and $w\in\{0,1,\dots,\bar{w}\}$. 
\end{prop}

We state some notation necessary for further discussions. For any
$\alpha\in\bar{E}\cup V$ and $n\in\{1,2,\dots\}$, let $p(n,\alpha)$
be 
\begin{equation}
p(n,\alpha):=\max\{w':w'\leq\bar{w},\alpha\in S_{n,w'}\}.\label{eq:p-n-alpha}
\end{equation}
In other words, $p(n,\alpha)$ is the index $w'$ such that $\alpha\in S_{n,w'}$
but $\alpha\notin S_{n,k}$ for all $k\in\{w'+1,\dots,\bar{w}\}$.
We make three assumptions listed below.
\begin{assumption}
\label{assu:to-start-subalg}(Start of Algorithm \ref{alg:subdiff-subalg})
Recall that at the start of Algorithm \ref{alg:subdiff-subalg}, $S_{n,w}\subset V_{4}\cup V_{5}$.
We make three assumptions.

\begin{enumerate}
\item Whenever $(n,w)$ is such that $w>1$ and $S_{n,w}\subset V_{4}\cup V_{5}$
so that Algorithm \ref{alg:subdiff-subalg} is invoked, each $\mathbf{z}_{i}^{n,w-1}\in\mathbf{X}$,
where $i\in V_{4}\cup V_{5}$, is such that $[\mathbf{z}_{i}^{n,w-1}]_{i}\in X_{i}$
is the optimizer to the problem 
\begin{equation}
\begin{array}{c}
\underset{z\in X_{i}}{\min}\frac{1}{2}\|[\bar{\mathbf{x}}-\mathbf{v}_{H}^{n,w-1}]_{i}-z\|^{2}+f_{i,n,w-1}^{*}(z).\end{array}\label{eq:multi-node-start}
\end{equation}
In other words, suppose $w_{i}\geq1$ is the largest $w'$ such that
$i\in S_{n,w'}$ and $i\notin S_{n,\tilde{w}}$ for all $\tilde{w}\in\{w'+1,w'+2,\dots,w-1\}$.
Then for all $\tilde{w}\in\{w_{i}+1,\dots,w-1\}$, and $\alpha\in S_{n,\tilde{w}}$,
the condition $\mathbf{v}\in H_{\alpha}$ implies $[\mathbf{v}]_{i}=0$.
\item Suppose that for all $i\in V_{4}\cup V_{5}$, $\tilde{w}\in\{p(n,i)+1,\dots,\bar{w}\}$
and $\alpha\in S_{n,\tilde{w}}$, the condition $\mathbf{v}\in H_{\alpha}$
implies $[\mathbf{v}]_{i}=0$. (This implies $\mathbf{x}_{i}^{n,p(n,i)}=\mathbf{x}_{i}^{n,\bar{w}}$.)
\item Suppose that $S_{n,1}=V_{4}\cup V_{5}$ for all $n>1$.
\end{enumerate}
\end{assumption}

With these assumptions, we are able to prove the following. Even though
the proof in \cite{Pang_sub_Dyk} for the analogue of Theorem \ref{thm:convergence}
below was for the case of Example \ref{exa:distrib-dyk}, the proofs
can be carried over in a straightforward manner. 
\begin{thm}
\label{thm:convergence} \cite{Pang_sub_Dyk}(Convergence to primal
minimizer) Consider Algorithm \ref{alg:Ext-Dyk}. Assume that the
problem \eqref{eq:general-framework} is feasible, and for all $n\geq1$,
$\bar{E}_{n}=[\cup_{w=1}^{\bar{w}}S_{n,w}]\cap\bar{E}$, and $[\cup_{w=1}^{\bar{w}}S_{n,w}]\supset V$.
Suppose also that Assumption \ref{assu:to-start-subalg} holds.

For the sequence $\{\mathbf{z}_{\alpha}^{n,w}\}_{{1\leq n<\infty\atop 0\leq w\leq\bar{w}}}\subset\mathbf{X}$
for each $\alpha\in\bar{E}\cup V$ generated by Algorithm \ref{alg:Ext-Dyk}
and the sequences $\{v_{H}^{n,w}\}_{{1\leq n<\infty\atop 0\leq w\leq\bar{w}}}\subset\mathbf{X}$
and $\{v_{A}^{n,w}\}_{{1\leq n<\infty\atop 0\leq w\leq\bar{w}}}\subset\mathbf{X}$
thus derived, we have:

\begin{enumerate}
\item [(i)]For all $n\geq1$ and $w_{1},w_{2}\in\{1,\dots,\bar{w}\}$ such
that $w_{1}\leq w_{2}$, 
\[
\begin{array}{c}
F^{n,w_{2}}(\mathbf{z}^{n,w_{2}})\geq F^{n,w_{1}}(\mathbf{z}^{n,w_{1}})+\frac{1}{2}\underset{w'=w_{1}+1}{\overset{w_{2}}{\sum}}\|\mathbf{v}_{A}^{n,w'}-\mathbf{v}_{A}^{n,w'-1}\|^{2}.\end{array}
\]
Hence the sum $\sum_{n=1}^{\infty}\sum_{w=1}^{\bar{w}}\|\mathbf{v}_{A}^{n,w}-\mathbf{v}_{A}^{n,w-1}\|^{2}$
is finite and $\{F_{n,\bar{w}}(\{\mathbf{z}_{\alpha}^{n,\bar{w}}\}_{\alpha\in\bar{E}\cup V})\}_{n=1}^{\infty}$
is nondecreasing.
\item [(ii)]There is a constant $C$ such that $\|\mathbf{v}_{A}^{n,w}\|^{2}\leq C$
for all $n\in\mathbb{N}$ and $w\in\{1,\dots,\bar{w}\}$. 
\item [(iii)]For all $i\in V_{3}\cup V_{4}$, $n\geq1$ and $w\in\{1,\dots,\bar{w}\}$,
the vectors $\mathbf{z}_{i}^{n,w}$ are bounded. 
\end{enumerate}
\end{thm}

Recall that by the optimality of $\mathbf{z}_{i}^{n,p(n,i)}$ in \eqref{eq:Dykstra-min-subpblm}
and Proposition \ref{prop:sparsity}, we have 
\begin{equation}
[\mathbf{z}_{i}^{n,p(n,i)}]_{i}=\arg\min_{z_{i}\in X_{i}}\frac{1}{2}\bigg\| z_{i}-\underbrace{\bigg([\bar{\mathbf{x}}]_{i}-\sum_{\beta\neq i}[\mathbf{z}_{\beta}^{n,p(n,i)}]_{i}\bigg)}_{=:\bar{x}_{i}^{n,p(n,i)}}\bigg\|^{2}+f_{i,n,p(n,i)}^{*}(z_{i}).\label{eq:small-orig-dual}
\end{equation}
We also have 
\begin{equation}
\begin{array}{c}
[\mathbf{x}^{n,p(n,i)}]_{i}=\underset{x_{i}\in X_{i}}{\arg\min}\frac{1}{2}\|x_{i}-\bar{x}_{i}^{n,p(n,i)}\|^{2}+f_{i,n,p(n,i)}(x_{i}).\end{array}\label{eq:small-orig-primal}
\end{equation}
To see that \eqref{eq:small-orig-primal} holds, there are three cases
to consider. The first case is when \eqref{eq:Dykstra-min-subpblm}
in Algorithm \ref{alg:Ext-Dyk} is solved, in which case one can check
that \eqref{eq:small-orig-primal} holds by looking at the $i$th
component in \eqref{eq:Dykstra-min-subpblm}. The second case is when
\eqref{eq:finw-design} in Algorithm \ref{alg:subdiff-subalg} holds
(which is equivalent to \eqref{eq:small-orig-primal}), and the last
case is the treatment for subdifferentiable functions in the analogue
of Algorithm \ref{alg:subdiff-subalg} in \cite{Pang_rate_D_Dyk}. 

Note that \eqref{eq:small-orig-dual} and \eqref{eq:small-orig-primal}
can be considered primal-dual pairs. (The more accurate primal-dual
pair is \eqref{eq:p-d-1-case}, but it is clear that we can change
the sign and add a constant to one of the problems to get the pair
\eqref{eq:p-d-1-case}.) 

Let the prox center $\bar{x}_{i}^{n,p(n,i)}$ be as marked in \eqref{eq:small-orig-dual},
and let $\hat{z}_{i}^{n,p(n,i)}\in X_{i}$ be 
\begin{equation}
\begin{array}{c}
\hat{z}_{i}^{n,p(n,i)}=\underset{\hat{z}_{i}\in X_{i}}{\arg\min}\frac{1}{2}\|\hat{z}_{i}-\bar{x}_{i}^{n,p(n,i)}\|^{2}+f_{i}^{*}(\hat{z}_{i}),\end{array}\label{eq:small-sub-dual}
\end{equation}
and let $\hat{\mathbf{z}}_{i}^{n,p(n,i)}\in X$ be such that $[\hat{\mathbf{z}}_{i}^{n,p(n,i)}]_{i}=\hat{z}_{i}^{n,p(n,i)}$,
and $[\hat{\mathbf{z}}_{i}^{n,p(n,i)}]_{j}=0$ if $j\neq i$. Note
that \eqref{eq:small-sub-dual} is distinct from \eqref{eq:small-orig-dual}.
If $e\in\bar{E}$, then $\hat{\mathbf{z}}_{e}^{n,p(n,e)}=\mathbf{z}_{e}^{n,p(n,e)}$.
Let $\hat{x}_{i}^{n,p(n,i)}\in X_{i}$ be found by the dual to \eqref{eq:small-sub-dual},
i.e.,
\begin{equation}
\begin{array}{c}
\hat{x}_{i}^{n,p(n,i)}=\underset{\hat{x}_{i}\in X_{i}}{\arg\min}\frac{1}{2}\|\hat{x}_{i}-\bar{x}_{i}^{n,p(n,i)}\|^{2}+f_{i}(\hat{x}_{i}),\end{array}\label{eq:small-sub-primal}
\end{equation}

The Moreau decomposition theorem can be used to prove that 
\begin{equation}
[\mathbf{x}^{n,p(n,i)}]_{i}+[\mathbf{z}_{i}^{n,p(n,i)}]_{i}=\hat{x}_{i}^{n,p(n,i)}+\hat{z}_{i}^{n,p(n,i)}=\bar{x}_{i}^{n,p(n,i)}.\label{eq:Moreau}
\end{equation}
Define $\Delta x_{i}\in X_{i}$ by 
\begin{equation}
\Delta x_{i}:=[\mathbf{x}^{n,p(n,i)}]_{i}-\hat{x}_{i}^{n,p(n,i)}\overset{\eqref{eq:Moreau}}{=}[\hat{\mathbf{z}}_{i}^{n,p(n,i)}]_{i}-[\mathbf{z}_{i}^{n,p(n,i)}]_{i}\in X_{i}.\label{eq:def-delta-z-i}
\end{equation}
Note that this value was called $\Delta z_{i}$ in \cite{Pang_rate_D_Dyk}.
We have 
\begin{eqnarray*}
 &  & \begin{array}{c}
\langle[\mathbf{x}^{n,p(n,i)}]_{i},[\mathbf{z}_{i}^{n,p(n,i)}]_{i}\rangle-\langle\hat{x}_{i}^{n,p(n,i)},[\hat{\mathbf{z}}_{i}^{n,p(n,i)}]_{i}\rangle\end{array}\\
 & \overset{\eqref{eq:def-delta-z-i}}{=} & \begin{array}{c}
\langle[\mathbf{x}^{n,p(n,i)}]_{i}-[\mathbf{z}_{i}^{n,p(n,i)}]_{i},\Delta z_{i}\rangle+\frac{1}{2}\|\Delta x_{i}\|^{2}.\end{array}
\end{eqnarray*}
Define $\hat{\mathbf{v}}_{A}^{n,w}\in\mathbf{X}$ like \eqref{eq_m:all_acronyms}
to be 
\begin{equation}
\begin{array}{c}
\hat{\mathbf{v}}_{A}^{n,w}:=\underset{\alpha\in V\cup\bar{E}}{\overset{\phantom{\alpha\in V\cup\bar{E}}}{\sum}}\hat{\mathbf{z}}_{i}^{n,w}.\end{array}\label{eq:v-hat}
\end{equation}
Let $h^{n,w}=-F^{n,w}(\mathbf{z}^{n,w})-(-F(\mathbf{z}^{*}))$, where
$\mathbf{z}^{*}\in\mathbf{X}^{|V\cup\bar{E}|}$ is any optimal solution,
and let $h^{n}$ be defined by 
\begin{equation}
h^{n}:=h^{n,0}=-F^{n,0}(\mathbf{z}^{n,0})-(-F(\mathbf{z}^{*})).\label{eq:def-h-n-values}
\end{equation}
Note that $h^{n}\geq0$. For the case $i\in V_{4}$, define $\Delta f_{i}\in\mathbb{R}$
to be 
\begin{equation}
\Delta f_{i}:=f_{i}([\mathbf{x}^{n,p(n,i)}]_{i})-f_{i,n,p(n,i)}([\mathbf{x}^{n,p(n,i)}]_{i}).\label{eq:def-delta-f-i}
\end{equation}
Formula \eqref{eq:h-a-n-w-lesser} implies $\Delta f_{i}\geq0$ for
all $i\in V_{4}$. We now recall some formulas proved in \cite{Pang_rate_D_Dyk}.
\begin{prop}
\cite{Pang_rate_D_Dyk} For each $\alpha\in V\cup\bar{E}$, the partial
subdifferential of $(-F)$ in the $\alpha$-th coordinate is, by \cite{Pang_rate_D_Dyk}
\begin{equation}
\partial(-F)_{\alpha}(\hat{\mathbf{z}}^{n,\bar{w}})\ni\underbrace{\sum_{\beta\neq\alpha}\hat{\mathbf{z}}_{\beta}^{n,\bar{w}}}_{\hat{\mathbf{t}}_{\alpha}}-\underbrace{\sum_{\beta\neq\alpha}\mathbf{z}_{\beta}^{n,p(n,\alpha)}}_{\mathbf{t}_{\alpha}}.\label{eq:partial-subdiff}
\end{equation}
Let $\hat{\mathbf{t}}_{\alpha}$ and $\mathbf{t}_{\alpha}$ be as
marked. Define $\bar{\mathbf{t}}_{\alpha}\in\mathbf{X}$ to be 
\begin{equation}
\begin{array}{c}
\bar{\mathbf{t}}_{\alpha}:=\underset{\beta\neq\alpha}{\overset{\phantom{\beta\neq\alpha}}{\sum}}\mathbf{z}_{\beta}^{n,\bar{w}}.\end{array}\label{eq:big-t-bar}
\end{equation}
Note the inequality $\|\mathbf{t}_{\alpha}-\bar{\mathbf{t}}_{\alpha}\|\leq\begin{array}{c}
\sqrt{\bar{w}}\sqrt{h^{n}-h^{n+1}}\end{array}$ proved in \cite{Pang_rate_D_Dyk}. Define $\mathbf{s}\in\mathbf{X}^{|V\cap\bar{E}|}$
by $\mathbf{s}_{\alpha}=\hat{\mathbf{t}}_{\alpha}-\mathbf{t}_{\alpha}$.
Let $\mathbf{z}^{*}\in\mathbf{X}^{|V\cup\bar{E}|}$ be any optimal
solution to \eqref{eq:general-dual}. From \eqref{eq:partial-subdiff},
we have $\mathbf{s}\in\partial(-F)(\hat{\mathbf{z}}^{n,\bar{w}})$,
and the techniques in \cite{Pang_rate_D_Dyk} give 
\begin{eqnarray}
 &  & \begin{array}{c}
-F(\hat{\mathbf{z}}^{n,\bar{w}})-[-F(\mathbf{z}^{*})]\leq\underset{\alpha\in V\cup\bar{E}}{\overset{\phantom{\alpha\in V\cup\bar{E}}}{\sum}}\big[[\|\bar{\mathbf{t}}_{\alpha}-\hat{\mathbf{t}}_{\alpha}\|+\|\mathbf{t}_{\alpha}-\bar{\mathbf{t}}_{\alpha}\|]\|\mathbf{z}_{\alpha}^{*}-\hat{\mathbf{z}}_{\alpha}^{n,\bar{w}}\|\big].\end{array}\nonumber \\
 & \leq & \begin{array}{c}
\underset{\alpha\in V\cup\bar{E}}{\overset{\phantom{\alpha\in V\cup\bar{E}}}{\sum}}\big[\|\bar{\mathbf{t}}_{\alpha}-\hat{\mathbf{t}}_{\alpha}\|\|\mathbf{z}_{\alpha}^{*}-\hat{\mathbf{z}}_{\alpha}^{n,\bar{w}}\|\big]+\sqrt{\bar{w}}\sqrt{h^{n}-h^{n+1}}\!\!\underset{\alpha\in V\cup\bar{E}}{\overset{\phantom{\alpha\in V\cup\bar{E}}}{\sum}}\|\mathbf{z}_{\alpha}^{*}-\hat{\mathbf{z}}_{\alpha}^{n,\bar{w}}\|.\end{array}\label{eq:subdiff-for-h}
\end{eqnarray}
From the fact that $\hat{\mathbf{z}}_{i}^{n,w}=\mathbf{z}_{i}^{n,w}$
for all $i\in V_{1}\cup V_{2}\cup V_{3}$, we have
\begin{eqnarray}
 &  & -[F^{n,\bar{w}}(\mathbf{z}^{n,\bar{w}})-F(\hat{\mathbf{z}}^{n,\bar{w}})]\label{eq:subdiff-part-2}\\
 & \leq & \sum_{i\in V_{4}\cup V_{5}}\big[\langle[\mathbf{x}^{n,p(n,i)}]_{i}-[\mathbf{z}_{i}^{n,p(n,i)}]_{i},\Delta x_{i}\rangle+\frac{1}{2}\|\Delta x_{i}\|^{2}+[f_{i}(\hat{x}_{i}^{n,p(n,i)})-f_{i,n,p(n,i)}([\mathbf{x}^{n,p(n,i)}]_{i})]\big]\nonumber \\
 &  & -\langle(\bar{\mathbf{x}}-\mathbf{v}_{A}^{n,\bar{w}}),(\mathbf{v}_{A}^{n,\bar{w}}-\hat{\mathbf{v}}_{A}^{n,\bar{w}})\rangle-\frac{1}{2}\|\sum_{i\in V_{4}\cup V_{5}}[\mathbf{z}_{i}^{n,\bar{w}}-\hat{\mathbf{z}}_{i}^{n,\bar{w}}]\|^{2}\nonumber \\
 & \leq & \sum_{i\in V_{4}\cup V_{5}}\big[\|[\mathbf{x}^{n,p(n,i)}]_{i}-[\mathbf{z}_{i}^{n,p(n,i)}]_{i}\|\|\Delta x_{i}\|+[f_{i}(\hat{x}_{i}^{n,p(n,i)})-f_{i,n,p(n,i)}([\mathbf{x}^{n,p(n,i)}]_{i})]\big]\nonumber \\
 &  & +\|\bar{\mathbf{x}}-\mathbf{v}_{A}^{n,\bar{w}}\|\|\mathbf{v}_{A}^{n,\bar{w}}-\hat{\mathbf{v}}_{A}^{n,\bar{w}}\|.\nonumber 
\end{eqnarray}
Also, the steps in \cite{Pang_rate_D_Dyk} give 
\begin{equation}
\|\Delta x_{i}\|\leq\sqrt{\Delta f_{i}}\text{ for all }i\in V_{4}.\label{eq:norm-leq-sqrt}
\end{equation}
\end{prop}

We also showed in \cite{Pang_rate_D_Dyk} that there is a bound $c'$
such that $\Delta f_{i}\leq c'$ for all $i\in V_{4}$, which implies
that $\Delta f_{i}\leq\sqrt{c'}\sqrt{\Delta f_{i}}$ for all $i\in V_{4}$.
Let $L$ be the Lipschitz constant of $f_{i}(\cdot)$. Thus
\begin{eqnarray*}
 &  & [f_{i}(\hat{x}_{i}^{n,p(n,i)})-f_{i,n,p(n,i)}([\mathbf{x}^{n,p(n,i)}]_{i})]\\
 & \overset{\eqref{eq:def-delta-f-i}}{=} & \Delta f_{i}+[f_{i}(\hat{x}_{i}^{n,p(n,i)})-f_{i}([\mathbf{x}^{n,p(n,i)}]_{i})]\\
 & \overset{\eqref{eq:def-delta-z-i}}{\leq} & \Delta f_{i}+L\|\Delta x_{i}\|\overset{\eqref{eq:norm-leq-sqrt}}{\leq}\Delta f_{i}+L\sqrt{\Delta f_{i}}\leq(\sqrt{c'}+L)\sqrt{\Delta f_{i}}\text{ for all }i\in V_{4}.
\end{eqnarray*}
So by the above inequality and \eqref{eq:norm-leq-sqrt}, we have\textbf{
\begin{align}
 & \begin{array}{c}
\underset{i\in V_{4}}{\sum}\big[\|[\mathbf{x}^{n,p(n,i)}]_{i}-[\mathbf{z}_{i}^{n,p(n,i)}]_{i}\|\|\Delta x_{i}\|+[f_{i}(\hat{x}_{i}^{n,p(n,i)})-f_{i,n,p(n,i)}([\mathbf{x}^{n,p(n,i)}]_{i})]\big]\end{array}\nonumber \\
\leq & \begin{array}{c}
\underset{i\in V_{4}}{\overset{\phantom{i\in V_{4}}}{\sum}}\big[[\|[\mathbf{x}^{n,p(n,i)}]_{i}-[\mathbf{z}_{i}^{n,p(n,i)}]_{i}\|+\sqrt{c'}+L]\sqrt{\Delta f_{i}}\big].\end{array}\label{eq:simplify-V4}
\end{align}
}

\section{Proof of convergence }

In this section, we present the proof of convergence rate for the
distributed Dykstra's algorithm. 

We need a plane geometry result for our proof. 
\begin{prop}
\label{prop:plane-geom}Refer to Figure \ref{fig:the-fig} for an
illustration. Consider two points $\bar{x}$ and $\hat{x}$, and let
$\hat{H}$ be the halfspace $\{\tilde{x}:\langle\bar{x}-\hat{x},\tilde{x}-\hat{x}\rangle\leq0\}$.
Suppose that $x$ is such that $\hat{x}$ lies in the halfspace $H:=\{\tilde{x}:\langle\bar{x}-x,\tilde{x}-x\rangle\leq0\}$.
Then $\|x-\hat{x}\|\leq\|\bar{x}-\hat{x}\|$, and $d(x,\hat{H})\geq\frac{\|x-\hat{x}\|^{2}}{\|\bar{x}-\hat{x}\|}$.
\end{prop}

\begin{proof}
It is clear that $\hat{x}\in H$ implies $\angle\bar{x}x\hat{x}\geq\pi/2$,
which is equivalent to $x$ being in the sphere with diameter $\|\bar{x}-\hat{x}\|$
and center $\frac{1}{2}(\bar{x}+\hat{x})$. Thus $\|x-\hat{x}\|\leq\|\bar{x}-\hat{x}\|$.
For a fixed value of $\|\hat{x}-x\|$, the smallest distance $d(x,\hat{H})$
occurs when $x$ lies on the boundary of the sphere. Let the projection
of $x$ onto the line segment connecting $\bar{x}$ and $\hat{x}$
be $x'$. The triangles $\bar{x}x\hat{x}$ and $xx'\hat{x}$ are similar,
which gives the lower bound $\frac{\|x-\hat{x}\|^{2}}{\|\bar{x}-\hat{x}\|}$
for $d$ as needed.
\end{proof}
\begin{figure}[!h]
\includegraphics[scale=0.2]{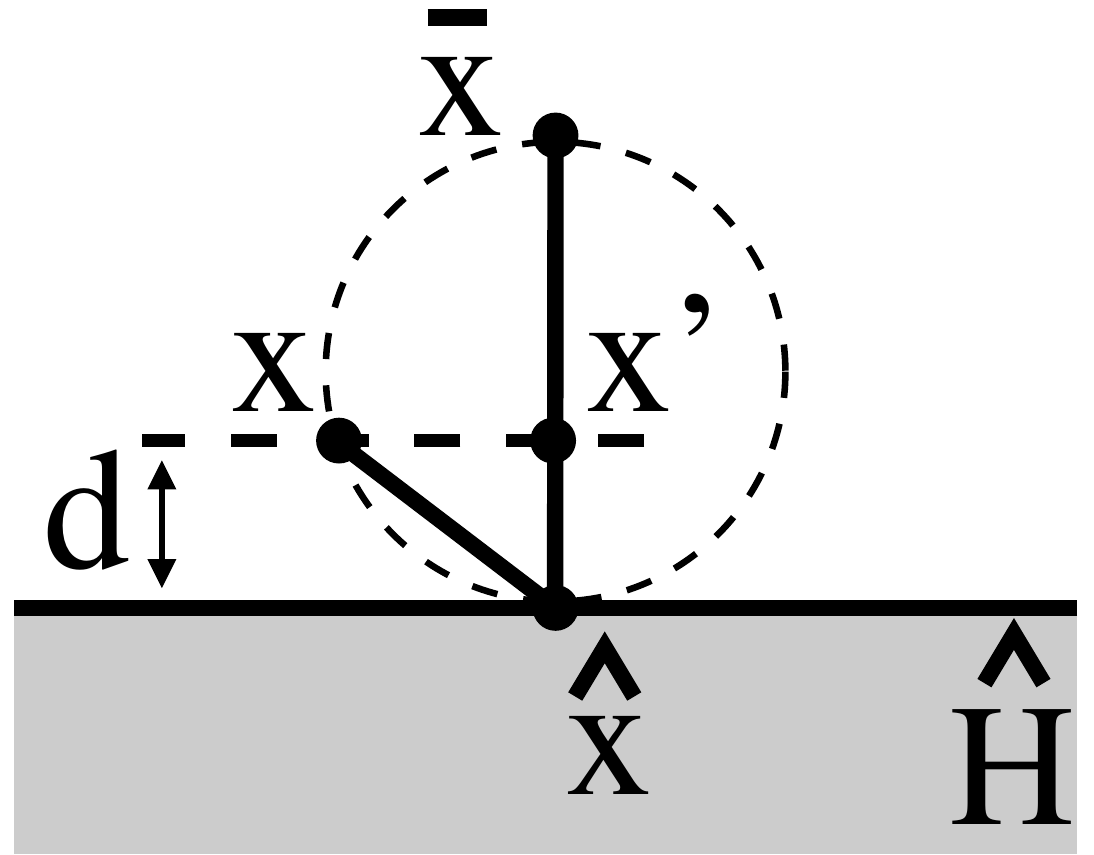}

\caption{\label{fig:the-fig}Diagram for Proposition \ref{prop:plane-geom}. }
\end{figure}

We write down the convergence result and its proof. 
\begin{thm}
Consider Algorithms \ref{alg:Ext-Dyk} and \ref{alg:subdiff-subalg}.
Suppose that Assumption \ref{assu:to-start-subalg} holds. Suppose
the iterates $\{\mathbf{z}_{\alpha}^{n,w}\}_{n,w}$ are bounded for
all $\alpha\in V\cup\bar{E}$, and that a minimizer $\mathbf{z}^{*}\in\mathbf{X}^{|V\cup\bar{E}|}$
exists. Suppose that the functions $g_{i}:X_{i}\to\mathbb{R}$ are
such that $\min_{x\in X_{i}}g_{i}(x)<0$ for all $i\in V_{5}$. Then
the values $\{h^{n}\}_{n}$ in \eqref{eq:def-h-n-values}, which measures
the rate at which the dual objective value converges to its optimal
value, converges to zero at an $O(1/n^{1/3})$ rate. 
\end{thm}

\begin{proof}
Recall the $\bar{x}_{i}^{n,p(n,i)}$ in \eqref{eq:small-orig-dual}.
We see that $\bar{x}_{i}^{n,p(n,i)}\overset{\eqref{eq:Moreau}}{=}[\mathbf{x}^{n,p(n,i)}]_{i}+[\mathbf{z}_{i}^{n,p(n,i)}]_{i}$
for all $i\in V$. Since $[\mathbf{x}^{n,p(n,i)}]_{i}$ is bounded
by Theorem \ref{thm:convergence}(ii) and \eqref{eq_m:all_acronyms},
and $[\mathbf{z}_{i}^{n,p(n,i)}]_{i}$ was assumed to be bounded,
$\bar{x}_{i}^{n,p(n,i)}$ is bounded. In view of \eqref{eq:small-sub-primal},
$\hat{x}_{i}^{n,p(n,i)}$ is the projection of $\bar{x}_{i}^{n,p(n,i)}$
onto $C_{i}=\{x:g_{i}(x)\leq0\}$. The projection of a bounded set
onto a closed convex set is bounded, so $\hat{x}_{i}^{n,p(n,i)}$
is also bounded. The point $[\mathbf{x}^{n,p(n,i)}]_{i}$ is the projection
of $\bar{x}_{i}^{n,p(n,i)}$ onto a superset of $C_{i}$, and so Proposition
\ref{prop:plane-geom} can be applied, with $\bar{x}$ being $\bar{x}_{i}^{n,p(n,i)}$,
$\hat{x}$ being $\hat{x}_{i}^{n,p(n,i)}$ and $x$ being $[\mathbf{x}^{n,p(n,i)}]_{i}$.
This shows that $\Delta x_{i}$ (defined in \eqref{eq:def-delta-z-i})
and $\hat{z}_{i}^{n,p(n,i)}\overset{\eqref{eq:def-delta-z-i}}{=}[\mathbf{z}_{i}^{n,p(n,i)}]_{i}+\Delta x_{i}$
are bounded for all $i\in V_{5}$. From here, we deduce that $\|[\mathbf{x}^{n,p(n,i)}]_{i}-[\mathbf{z}_{i}^{n,p(n,i)}]_{i}\|$
and $\|\bar{\mathbf{x}}-\mathbf{v}_{A}^{n,\bar{w}}\|$ are bounded
for all $i\in V_{5}$ and $\alpha\in V\cup\bar{E}$. Let $\mathbf{z}^{*}\in\mathbf{X}^{|V\cup\bar{E}|}$
be any optimal solution to \eqref{eq:general-dual}. We then have
$\|\mathbf{z}_{\alpha}^{*}-\hat{\mathbf{z}}_{\alpha}^{n,w}\|$ being
bounded for all $\alpha\in V\cup\bar{E}$.

The quantity $\|\mathbf{v}_{A}^{n,\bar{w}}-\hat{\mathbf{v}}_{A}^{n,\bar{w}}\|$
is bounded by a multiple of $\sum_{i\in V_{5}}\|\Delta x_{i}\|+\sum_{i\in V_{4}}\sqrt{\Delta f_{i}}$
by \eqref{eq:def-delta-z-i}, \eqref{eq:v-hat}, \eqref{eq_m:all_acronyms}
and \eqref{eq:norm-leq-sqrt}. We note that $\hat{x}_{i}^{n,p(n,i)}\in C_{i}$
by \eqref{eq:small-sub-primal} and the discussion after \eqref{eq_m:h-a-n-w},
and $[\mathbf{x}^{n,p(n,i)}]_{i}\in\{x:g_{i,n,w}(x)\leq0\}$ by \eqref{eq:small-orig-primal}
and the discussion after \eqref{eq_m:h-a-n-w}, so $f_{i}(\hat{x}_{i}^{n,p(n,i)})=f_{i,n,p(n,i)}([\mathbf{x}^{n,p(n,i)}]_{i})=0$
for all $i\in V_{5}$, which gives 

\textbf{
\begin{align}
 & \begin{array}{c}
\underset{i\in V_{5}}{\sum}\big[\|[\mathbf{x}^{n,p(n,i)}]_{i}-[\mathbf{z}_{i}^{n,p(n,i)}]_{i}\|\|\Delta x_{i}\|+[f_{i}(\hat{x}_{i}^{n,p(n,i)})-f_{i,n,p(n,i)}([\mathbf{x}^{n,p(n,i)}]_{i})]\big]\end{array}\nonumber \\
= & \begin{array}{c}
\underset{i\in V_{5}}{\sum}\big[\|[\mathbf{x}^{n,p(n,i)}]_{i}-[\mathbf{z}_{i}^{n,p(n,i)}]_{i}\|\|\Delta x_{i}\|\big]\end{array}\label{eq:simplify-V5}
\end{align}
}Then there are constants $c_{1}$, $c_{2}$ and $c_{3}$ such that
\begin{eqnarray}
h^{n+1} & \overset{\eqref{eq:def-h-n-values}}{=} & \begin{array}{c}
-[F^{n,\bar{w}}(\mathbf{z}^{n,\bar{w}})-F(\hat{\mathbf{z}}^{n,\bar{w}})]-[F(\hat{\mathbf{z}}^{n,\bar{w}})-F(\mathbf{z}^{*})]\end{array}\nonumber \\
 & \overset{\eqref{eq:subdiff-for-h},\eqref{eq:subdiff-part-2},\eqref{eq:simplify-V4},\eqref{eq:simplify-V5}}{\leq} & \begin{array}{c}
c_{1}\sqrt{h^{n}-h^{n+1}}+\underset{i\in V_{4}}{\overset{\phantom{V_{5}}}{\sum}}c_{2}\sqrt{\Delta f_{i}}+\underset{i\in V_{5}}{\overset{\phantom{V_{5}}}{\sum}}c_{3}\|\Delta x_{i}\|.\end{array}\label{eq:h-n-late}
\end{eqnarray}
The next step is to show how $\|\Delta x_{i}\|$ can be related to
the decrease in $\{h^{n,w}\}$. We want to show that 
\begin{equation}
\begin{array}{c}
h^{n+1,1}\leq h^{n+1,0}-c_{4}\underset{i\in V_{4}}{\overset{\phantom{i\in V_{4}}}{\sum}[}\Delta f_{i}]^{2}-\underset{i\in V_{5}}{\overset{\phantom{i\in V_{4}}}{\sum}}c_{5}\|\Delta x_{i}\|^{4}.\end{array}\label{eq:h_n_plus_1_recur}
\end{equation}
The dual objective function that we have at iteration $(n,w)$ is
\eqref{eq:h-def-1}. Note that $S_{n+1,1}$ satisfies $S_{n+1,1}\subset V$
by Assumption \ref{assu:to-start-subalg}(3), and $\delta_{H_{e}}(\mathbf{z}_{e}^{n+1,1})$
is already finite for all $e\in\bar{E}$, and by the sparsity of the
$\mathbf{z}_{i}$'s (Proposition \ref{prop:sparsity}), maximizing
\eqref{eq:h-def-1} is equivalent to minimizing 
\begin{equation}
\underset{[\mathbf{z}_{j}]_{j}\in X_{j},j\in V_{4}\cup V_{5}}{\min}\underset{i\in S_{n,w}}{\sum}\underbrace{\Big[\frac{1}{2}\|\bar{x}_{i}^{n+1,0}-[\mathbf{z}_{i}]_{i}\|^{2}+f_{i,n+1,1}^{*}([\mathbf{z}_{i}]_{i})\Big]},\label{eq:dual-fn-sum-form}
\end{equation}
where $\bar{x}_{i}^{n+1,0}$ be $[\bar{\mathbf{x}}]_{i}-\sum_{\beta\neq i}\mathbf{z}_{\beta}^{n+1,0}=\mathbf{x}_{i}^{n+1,0}+[\mathbf{z}_{i}^{n+1,0}]_{i}$.
We can look separately at the dual problems underbraced in the above
problem. The dual of these problems are, up to a sign change and a
constant, 
\[
\begin{array}{c}
\underset{x_{i}\in X_{i}}{\min}\frac{1}{2}\|\bar{x}_{i}^{n+1,0}-x_{i}\|^{2}+f_{i,n+1,1}(x_{i}).\end{array}
\]
We now treat the case of $i\in V_{5}$. Recall that \begin{subequations}\label{eq_m:proj_subpblms}
\begin{eqnarray}
[\mathbf{x}^{n+1,0}]_{i} & = & \begin{array}{c}
\arg\min_{x_{i}\in X_{i}}\frac{1}{2}\|x_{i}-\bar{x}_{i}^{n+1,0}\|^{2}+f_{i,n+1,0}(x_{i}),\end{array}\label{eq:proj-subpblms-1}\\
\hat{x}_{i}^{n+1,0} & = & \begin{array}{c}
\arg\min_{\hat{x}_{i}\in X_{i}}\frac{1}{2}\|\hat{x}_{i}-\bar{x}_{i}^{n+1,0}\|^{2}+f_{i}(\hat{x}_{i}),\end{array}\label{eq:proj-subpblms-2}\\{}
[\mathbf{x}^{n+1,1}]_{i} & = & \begin{array}{c}
\arg\min_{x_{i}\in X_{i}}\frac{1}{2}\|x_{i}-\bar{x}_{i}^{n+1,0}\|^{2}+f_{i,n+1,1}(x_{i}),\end{array}\label{eq:proj-subpblms-3}
\end{eqnarray}
\end{subequations}Since $f_{i,n+1,0}$, $f_{i}$ and $f_{i,n+1,1}$
are indicator functions, the primal objective values of the problems
are $\frac{1}{2}\|\bar{x}_{i}^{n+1,0}-[\mathbf{x}^{n+1,0}]_{i}\|^{2}$,
$\frac{1}{2}\|\bar{x}_{i}^{n+1,0}-\hat{x}_{i}^{n+1,0}\|^{2}$ and
$\frac{1}{2}\|\bar{x}_{i}^{n+1,0}-[\mathbf{x}^{n+1,1}]_{i}\|^{2}$
respectively. Let the halfspace $\hat{H}_{i}^{n+1,0}$ be 
\[
\hat{H}_{i}^{n+1}=\{x\in X_{i}:\langle\bar{x}_{i}^{n+1,0}-\hat{x}_{i}^{n+1,0},x-\hat{x}_{i}^{n+1,0}\rangle\}.
\]
In view of Assumption \ref{assu:to-start-subalg}(2), we have $\Delta x_{i}\overset{\eqref{eq:def-delta-z-i}}{=}[\mathbf{x}^{n+1,0}]_{i}-\hat{x}_{i}^{n+1,0}$.
Define $H_{i}^{n+1}$ to be the halfspace separating $[\mathbf{x}^{n+1,0}]_{i}$
from the set $C_{i}$ formed in Algorithm \ref{alg:subdiff-subalg}.
Recall that $\bar{x}_{i}^{n+1,0}$ and $[\mathbf{x}^{n+1,0}]_{i}$
are bounded inside some set for all $i\in V$, say $R$. Let $\gamma=\sup\{\|s\|:s\in\partial g_{i}(x),i\in V\}$
and $\tilde{\gamma}=\inf\{\|s\|:s\in\partial g_{i}(x),i\in V\}$,
which both have to be finite numbers by arguments in Lemma \ref{lem:3_pts}.
The boundedness of $\|\bar{x}_{i}^{n+1,0}-\hat{x}_{i}^{n+1,0}\|$
as mentioned earlier shows that there is a constant $c_{5}$ such
that for all $n\geq0$, $c_{5}>0$ and $\frac{\gamma}{\tilde{\gamma}\|\bar{x}_{i}^{n+1,0}-\hat{x}_{i}^{n+1,0}\|}\geq c_{5}$.
We have 
\begin{eqnarray}
\begin{array}{c}
d([\mathbf{x}^{n+1,0}]_{i},H_{i}^{n+1})\end{array} & \overset{\scriptsize{\text{Lem. \ref{lem:3_pts}(4)}}}{\geq} & \begin{array}{c}
\frac{1}{\tilde{\gamma}}g_{i}([\mathbf{x}^{n+1,0}]_{i})\overset{\scriptsize{\text{Lem \ref{lem:3_pts}(3)}}}{\underset{\phantom{\scriptsize{\text{Lem \ref{lem:3_pts}(3)}}}}{\geq}}\frac{\gamma}{\tilde{\gamma}}d(x_{i}^{n+1,0},\hat{H}_{i}^{n+1,0})\end{array}\nonumber \\
 & \overset{\scriptsize{\text{Prop \ref{prop:plane-geom}}}}{\geq} & \begin{array}{c}
\frac{\gamma}{\tilde{\gamma}}\frac{\|\Delta x_{i}\|^{2}}{\|\bar{x}_{i}^{n+1,0}-\hat{x}_{i}^{n+1,0}\|_{\phantom{2}}}\geq c_{5}\|\Delta x_{i}\|^{2}.\end{array}\label{eq:est-dist-recur}
\end{eqnarray}
By Lemma \ref{lem:d_sqr_increase}, we have
\begin{eqnarray}
\begin{array}{c}
\|\bar{x}_{i}^{n+1,0}-[\mathbf{x}^{n+1,1}]_{i}\|^{2}\end{array} & \overset{\scriptsize{\text{Lem \ref{lem:d_sqr_increase}}}}{\geq} & \begin{array}{c}
\|\bar{x}_{i}^{n+1,0}-[\mathbf{x}^{n+1,0}]_{i}\|^{2}+d([\mathbf{x}^{n+1,0}]_{i},H^{n+1})^{2}\end{array}\nonumber \\
 & \overset{\eqref{eq:est-dist-recur}}{\geq} & \begin{array}{c}
\|\bar{x}_{i}^{n+1,0}-[\mathbf{x}^{n+1,0}]_{i}\|^{2}+c_{5}\|\Delta x_{i}\|^{4}.\end{array}\label{eq:dist-inc-step}
\end{eqnarray}
Note that $f_{i,n+1,1}([\mathbf{x}^{n+1,1}]_{i})=f_{i,n+1,0}([\mathbf{x}^{n+1,0}]_{i})=0$,
which gives the values of the optimization problems in \eqref{eq_m:proj_subpblms}.
The strong duality between the primal problems of the type \eqref{eq_m:proj_subpblms}
and its dual (of the type \eqref{eq:p-d-1-case}) implies that 
\begin{eqnarray}
 &  & \begin{array}{c}
\frac{1}{2}\|\bar{x}_{i}^{n+1,0}-[\mathbf{x}^{n+1,1}]_{i}\|^{2}\end{array}\label{eq:proj-link-1}\\
 & = & \begin{array}{c}
\frac{1}{2}\|\bar{x}_{i}^{n+1,0}-[\mathbf{x}^{n+1,1}]_{i}\|^{2}+f_{i,n+1,1}([\mathbf{x}^{n+1,1}]_{i})\end{array}\nonumber \\
 & = & \begin{array}{c}
\frac{1}{2}\|\bar{x}_{i}^{n+1,0}\|^{2}-\frac{1}{2}\|\bar{x}_{i}^{n+1,0}-[\mathbf{z}_{i}^{n+1,1}]_{i}\|^{2}-f_{i,n+1,1}^{*}([\mathbf{z}_{i}^{n+1,1}]_{i})\end{array}.\nonumber 
\end{eqnarray}
There is a similar equation for $[\mathbf{x}^{n+1,0}]_{i}$. Combining
\eqref{eq:dist-inc-step} and \eqref{eq:proj-link-1} gives us 
\begin{eqnarray}
 &  & \begin{array}{c}
\frac{1}{2}\|\bar{x}_{i}^{n+1,0}-[\mathbf{z}_{i}^{n+1,1}]_{i}\|^{2}+f_{i,n+1,1}^{*}([\mathbf{z}_{i}^{n+1,1}]_{i})\end{array}\label{eq:component-for-V5}\\
 & \overset{\eqref{eq:dist-inc-step},\eqref{eq:proj-link-1}}{\leq} & \begin{array}{c}
\frac{1}{2}\|\bar{x}_{i}^{n+1,0}-[\mathbf{z}_{i}^{n+1,0}]_{i}\|^{2}+f_{i,n+1,0}^{*}([\mathbf{z}_{i}^{n+1,0}]_{i})-c_{5}\|\Delta x_{i}\|^{4}\text{ for all }i\in V_{5}.\end{array}\nonumber 
\end{eqnarray}
For the case when $i\in V_{4}$, an inequality similar to \eqref{eq:component-for-V5}
was obtained in \cite{Pang_rate_D_Dyk}, but the last term would be
replaced by $c_{4}[\Delta f_{i}]^{2}$ instead, where $c_{4}>0$ is
some constant. Summing up the inequalities of the form \eqref{eq:component-for-V5}
over all $i$ (note that if $i\in V_{1}\cup V_{2}\cup V_{3}$, $[\mathbf{z}_{i}^{n+1,0}]_{i}=[\mathbf{z}_{i}^{n+1,1}]_{i}$)
and that the dual \eqref{eq:h-def-1} can be written as the sum \eqref{eq:dual-fn-sum-form},
we have \eqref{eq:h_n_plus_1_recur} as needed.  

We now consider 2 cases:

\textbf{Case 1:} If $h^{n+1}\leq2c_{1}\sqrt{h^{n}-h^{n+1}}$, then
choose $\bar{c}>0$ such that $h^{n}\leq\bar{c}$ for all $n\geq1$.
So
\begin{equation}
\begin{array}{c}
h^{n}\overset{\scriptsize{\text{Case 1}}}{\geq}h^{n+1}+\frac{1}{4c_{1}^{2}}[h^{n+1}]^{2}\geq h^{n+2}+\frac{1}{4c_{1}^{2}}[h^{n+2}]^{2}\geq h^{n+2}+\frac{1}{4\bar{c}^{2}c_{1}^{2}}[h^{n+2}]^{4}.\end{array}\label{eq:recur_t1}
\end{equation}

\textbf{Case 2:} If $h^{n+1}\geq2c_{1}\sqrt{h^{n}-h^{n+1}}$, then\textbf{}
\begin{equation}
\begin{array}{c}
\underset{i\in V_{4}}{\overset{\phantom{i\in V_{4}}}{\sum}}c_{2}\sqrt{\Delta f_{i}}+\underset{i\in V_{5}}{\overset{\phantom{i\in V_{4}}}{\sum}}c_{3}\|\Delta x_{i}\|\overset{\eqref{eq:h-n-late}}{\geq}h^{n+1}-c_{1}\sqrt{h^{n}-h^{n+1}}\overset{\scriptsize{\text{Case 2}}}{\geq}\frac{1}{2}h^{n+1}.\end{array}\label{eq:low_power_surd}
\end{equation}
By the power means inequality, we have 
\begin{equation}
\begin{array}{c}
c_{4}\underset{i\in V_{4}}{\overset{\phantom{1\in V_{5}}}{\sum}[}\Delta f_{i}]^{2}+c_{5}\underset{i\in V_{5}}{\sum}\|\Delta x_{i}\|^{4}\geq\frac{1}{(|V_{4}|+|V_{5}|)^{3}}\Big[c_{4}^{1/4}\underset{i\in V_{4}}{\overset{\phantom{1\in V_{5}}}{\sum}}\sqrt{\Delta f_{i}}+c_{5}^{1/4}\underset{i\in V_{5}}{\sum}\|\Delta x_{i}\|\Big]^{4}.\end{array}\label{eq:high_power_surd}
\end{equation}
Incorporating \eqref{eq:low_power_surd} into \eqref{eq:high_power_surd}
shows that there is some $c_{6}>0$ such that
\begin{equation}
\begin{array}{c}
c_{4}\underset{i\in V_{4}}{\overset{\phantom{1\in V_{5}}}{\sum}[}\Delta f_{i}]^{2}+c_{5}\underset{i\in V_{5}}{\sum}\|\Delta x_{i}\|^{4}\geq c_{6}(h^{n+1})^{4}\end{array}\label{eq:ineq_h_4}
\end{equation}
Hence 
\begin{eqnarray}
\begin{array}{c}
h^{n+2}\end{array} & \leq & \begin{array}{c}
h^{n+1,1}\overset{\eqref{eq:h_n_plus_1_recur}}{\leq}h^{n+1,0}-c_{4}\underset{i\in V_{4}}{\overset{\phantom{1\in V_{5}}}{\sum}}[\Delta f_{i}]^{2}-c_{5}\underset{i\in V_{5}}{\sum}\|\Delta x_{i}\|^{4}\end{array}\nonumber \\
 & \overset{\eqref{eq:ineq_h_4}}{\leq} & \begin{array}{c}
h^{n+1}-c_{6}[h^{n+1}]^{4}\leq h^{n}-c_{6}[h^{n+2}]^{4}.\end{array}\label{eq:recur_t2}
\end{eqnarray}
Combining with \eqref{eq:recur_t1}, we have $h^{n+2}\leq h^{n}-\min\{c_{6},4\bar{c}^{2}c_{1}^{2}\}[h^{n+2}]^{4}$.
By Lemma \ref{lem:Beck-recur}(2), this recurrence guarantees a $O(1/n^{1/3})$
convergence of $\{h^{n}\}_{n}$. 
\end{proof}

\section{Numerical experiments}

In this section, we present the results of our numerical experiments
to verify the effectiveness of Algorithm \ref{alg:Ext-Dyk}.

We conduct 4 different sets of numerical experiments, and we now explain
their common features. Just like in \cite{Pang_rate_D_Dyk}, we look
at the graph where $|V|=5$ and $E=\{\{1,2\},\{1,3\},\{1,4\},\{1,5\}\}$.
We look at the setting of Example \ref{exa:distrib-dyk} where $X_{i}=\mathbb{R}^{m}$
and $m=10$ for all $i\in V$, and look at hyperplanes of the form
\[
H_{(i,j)}=\{\mathbf{{x}}\in\mathbf{{X}}:\mathbf{x}_{i}=\mathbf{\mathbf{x}}_{j}\}
\]
instead of the hyperplanes $H_{((i,j),k)}$ defined in \eqref{eq:H-alpha-subspaces}
to simplify computations. Our code is equivalent to $\bar{w}=8$ with
\begin{align*}
 & S_{n,1}=\{(1,2)\},S_{n,2}=\{1,2\},S_{n,3}=\{(1,3)\},S_{n,4}=\{1,3\},\\
 & S_{n,5}=\{(1,4)\},S_{n,6}=\{1,4\},S_{n,7}=\{(1,5)\},\text{ and }S_{n,8}=\{1,5\}.
\end{align*}
Let $\mathbf{e}$ be \texttt{ones(m,1)}. First, we find $\{v_{i}\}_{i\in V}$
and $\bar{x}$ such that $\sum_{i\in V}v_{i}+|V|(\mathbf{e}-\bar{x})=0$.
We then find closed convex functions $f_{i}(\cdot)$ such that $v_{i}\in\partial f_{i}(\mathbf{e})$.
It is clear from the KKT conditions that $\mathbf{e}$ is the primal
optimum solution to \eqref{eq:distrib-dyk-primal-pblm} if $[\bar{\mathbf{x}}]_{i}$
are all equal to $\bar{x}$ for all $i\in V$. 

The $f_{i}(\cdot)$ can be defined as either smooth or nonsmooth functions,
or as the indicator functions of level sets of smooth or nonsmooth
functions. They are described using some Matlab functions below. 
\begin{itemize}
\item [(F-S)]$f_{i}(x):=\frac{1}{2}x^{T}A_{i}x+b_{i}^{T}x+c_{i}$, where
$A_{i}$ is of the form $vv^{T}+rI$, where $v$ is generated by \texttt{rand(m,1)},
$r$ is generated by \texttt{rand(1)}. $b_{i}$ is chosen to be such
that $v_{i}=\nabla f(\mathbf{e})$, and $c_{i}=0$. 
\item [(F-NS)]$f_{i}(x):=\max\{f_{i,1}(x),f_{i,2}(x)\}$, where $f_{i,j}(x):=\frac{1}{2}x^{T}A_{i}x+b_{i,j}^{T}x+c_{i,j}$
for $j\in\{1,2\}$, $A_{i}$ is of the form $vv^{T}+rI$, where $v$
is generated by \texttt{rand(m,1)}, $r$ is generated by \texttt{rand(1)},
$b_{i,1}$ and $b_{i,2}$ are chosen such that $v_{i}=\frac{1}{2}[\nabla f_{i,1}(\mathbf{e})+\nabla f_{i,2}(\mathbf{e})]$
but $v_{i}$ is neither $\nabla f_{i,1}(\mathbf{e})$ nor $\nabla f_{i,2}(\mathbf{e})$,
and $c_{i,1}$ and $c_{i,2}$ are chosen such that $f_{i,1}(\mathbf{e})=f_{i,2}(\mathbf{e})$. 
\item [(LS-S)]$f_{i}(\cdot)=\delta_{\{x:g_{i}(x)\leq0\}}(\cdot)$, where
$g_{i}(x):=\frac{1}{2}x^{T}A_{i}x+b_{i}^{T}x+c_{i}$, $A_{i}$ is
of the form $vv^{T}+rI$, where $v$ is generated by \texttt{rand(m,1)},
$r$ is generated by \texttt{rand(1)}, and $b_{i}$ and $c_{i}$ are
chosen such that $g_{i}(\mathbf{e})=0$ and $v_{i}=\nabla g_{i}(\mathbf{e})$. 
\item [(LS-NS)]$f_{i}(\cdot)=\delta_{\{x:g_{i}(x)\leq0\}}(\cdot)$, where
$g_{i}(x):=\max\{g_{i,1}(x),g_{i,2}(x)\}$, $g_{i,j}(x):=\frac{1}{2}x^{T}A_{i,j}x+b_{i,j}^{T}x+c_{i,j}$
for $j\in\{1,2\}$, $A_{i,1}$ and $A_{i,2}$ are of the form $vv^{T}+rI$,
where $v$ is generated by \texttt{rand(m,1)}, $r$ is generated by
\texttt{rand(1)}, $b_{i,1}$ and $b_{i,2}$ are chosen such that $v_{i}=\frac{1}{2}[\nabla g_{i,1}(\mathbf{e})+\nabla g_{i,2}(\mathbf{e})]$
but $v_{i}$ is neither $\nabla g_{i,1}(\mathbf{e})$ nor $\nabla g_{i,2}(\mathbf{e})$,
and $g_{i,1}(\mathbf{e})=g_{i,2}(\mathbf{e})=0$. 
\end{itemize}
Note that in (F-S) and (LS-NS), the $b_{i,1}$ and $b_{i,2}$, as
well as $c_{i,1}$ and $c_{i,2}$ are not uniquely defined. We refer
to the source code to see how they are defined. For all the experiments,
we investigate the convergence behavior of $\frac{1}{2}\|x-x^{*}\|^{2}$
and the \textbf{\uline{duality gap}} defined by 
\[
\begin{array}{c}
\frac{1}{2}\|x^{n,w}\|^{2}+\underset{i\in V}{\sum}f_{i,n,w}^{*}([\mathbf{z}_{i}^{n,w}]_{i})-\Big[\frac{1}{2}\|x^{*}\|^{2}+\underset{i\in V}{\sum}f_{i}^{*}([\mathbf{z}_{i}^{*}]_{i})\Big],\end{array}
\]
where $(x^{*},\mathbf{z}^{*})$ is a dual optimal solution. It is
known that the duality gap is an upper bound for $\frac{1}{2}\|x-x^{*}\|^{2}$,
something which is verified in all our experiments. 

For the first set of experiments, we choose $f_{i}(\cdot)$ such that
$f_{i}(\cdot)$ are of the form (F-S) for all $i\in V$. 

For the second set of experiments, we choose $f_{i}(\cdot)$ to be
of the form (LS-S) for $i\in\{2,3\}$, and $f_{i}(\cdot)$ to be of
the form (F-S) for $i\in\{1,4,5\}$.

For the third set of experiments, we choose $f_{i}(\cdot)$ to be
of the form (LS-NS) for $i\in\{2,3\}$, and $f_{i}(\cdot)$ to be
of the form (F-S) for $i\in\{1,4,5\}$.

For the last set of experiments, we choose $f_{i}(\cdot)$ to be of
the form (LS-NS) for $i\in\{2,3\}$, and $f_{i}(\cdot)$ to be of
the form (F-NS) for $i\in\{1,4,5\}$.

In all the sets of experiments, we experiment over the cases when
all the $f_{i}(\cdot)$ marked to be in (F-S) or (F-NS) are either
all treated as subdifferentiable functions, or all treated as proximable
functions (i.e., either $V=V_{4}$, or $V=V_{1}$), and investigate
the behavior of both the duality gap and $\frac{1}{2}\|x-x^{*}\|^{2}$. 

We now elaborate on Figure \ref{fig:exp-1-2}. The two diagrams in
Figure \ref{fig:exp-1-2} show semi-log plots for the values of the
duality gap and $\frac{1}{2}\|x-x^{*}\|^{2}$ when the functions are
either all treated as subdifferentiable functions, or all treated
as proximable functions, with the first diagram corresponding to experiment
1 and the second diagram corresponding to experiment 2. There is a
(relatively fast) linear convergence of all values for the first set
of experiments, and a (relatively slow) linear convergence for the
second set of experiments. The former is consistent with the theory
in \cite{Pang_rate_D_Dyk}, while the latter is much better than the
$O(1/k^{1/3})$ rate that this paper suggests.

\begin{figure}[h]
\includegraphics[scale=0.4]{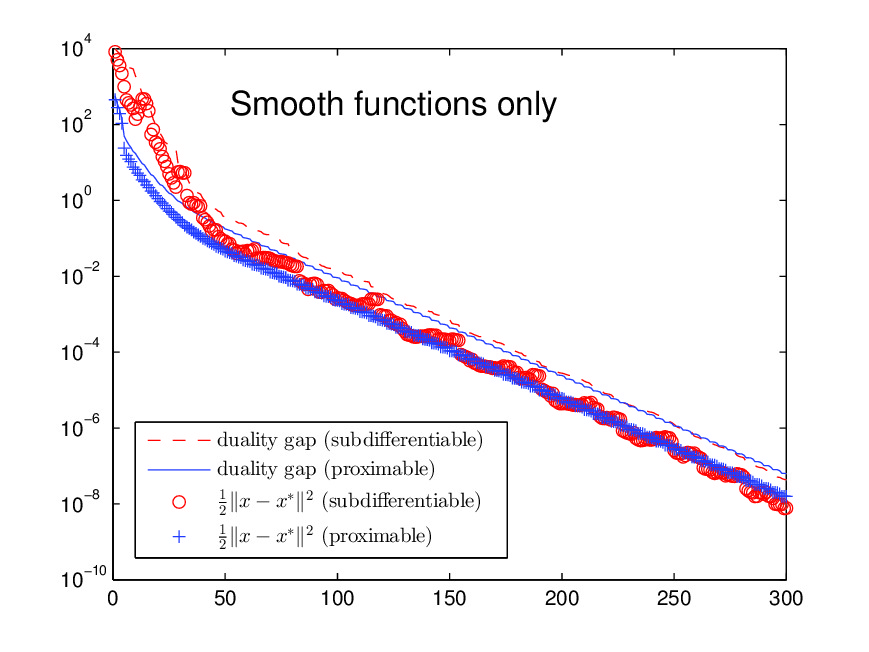}\includegraphics[scale=0.4]{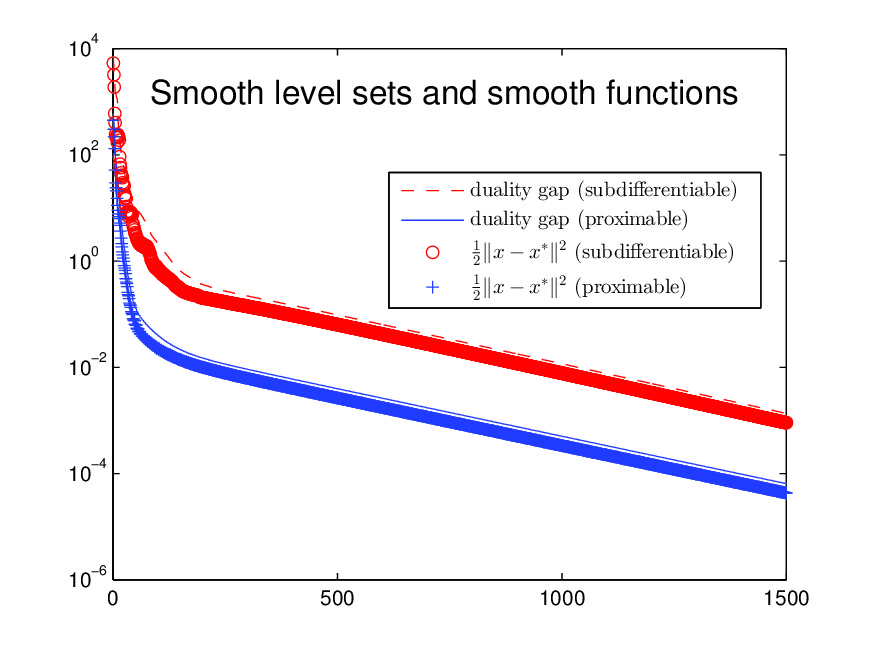}

\includegraphics{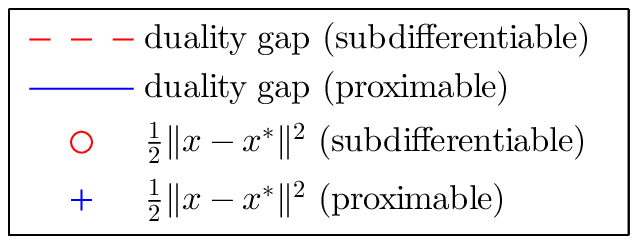}

\caption{\label{fig:exp-1-2}Semi-log plots for Experiments 1 and 2. Commentary
in main text. A magnified legend for both diagrams, which is also
the legend for the first diagrams in Figures \ref{fig:exp-3} and
\ref{fig:exp-4}, is also presented.}
\end{figure}

We now elaborate on Figure \ref{fig:exp-3}, which describes a typical
output from the third set of experiments. In the second diagram, a
plot of the reciprocal of the duality gaps for the cases when we treat
the smooth functions as proximable and subdifferentiable functions
gives straight lines, which shows an $O(1/k)$ convergence of the
duality gap. This is better than the $O(1/k^{1/3})$ rate proved in
this paper. In the third and fourth diagrams, the plots of $[\frac{1}{2}\|x-x^{*}\|^{2}]^{-1/2}$
look like a union of straight lines, which shows the $O(1/k^{2})$
convergence of $\frac{1}{2}\|x-x^{*}\|^{2}$. This cannot yet be explained
by the theory in both this paper and \cite{Pang_rate_D_Dyk}, where
the upper bound on the convergence rate we have is $O(1/k^{1/3})$.
There is also no noticeable performance improvement if we treat the
smooth functions as a proximable function instead of a subdifferentiable
function for both the duality gap and $\frac{1}{2}\|x-x^{*}\|^{2}$. 

\begin{figure}[h]
\includegraphics[scale=0.4]{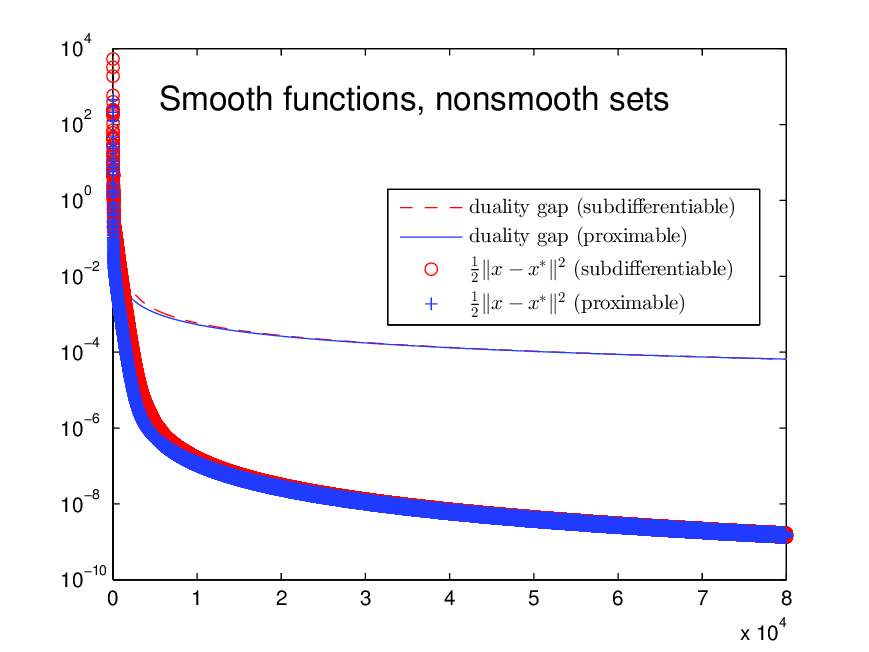}\includegraphics[scale=0.4]{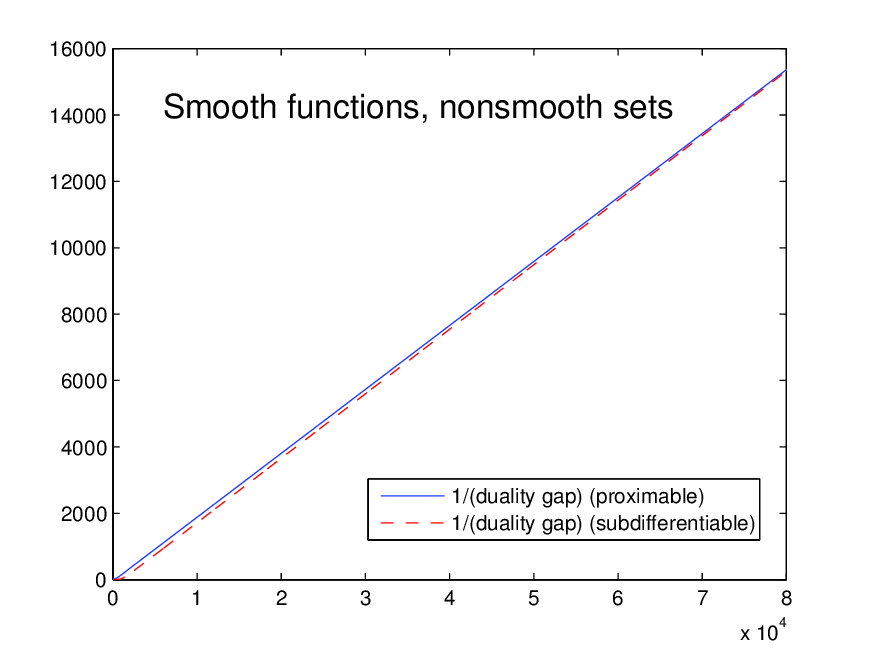}

\includegraphics[scale=0.4]{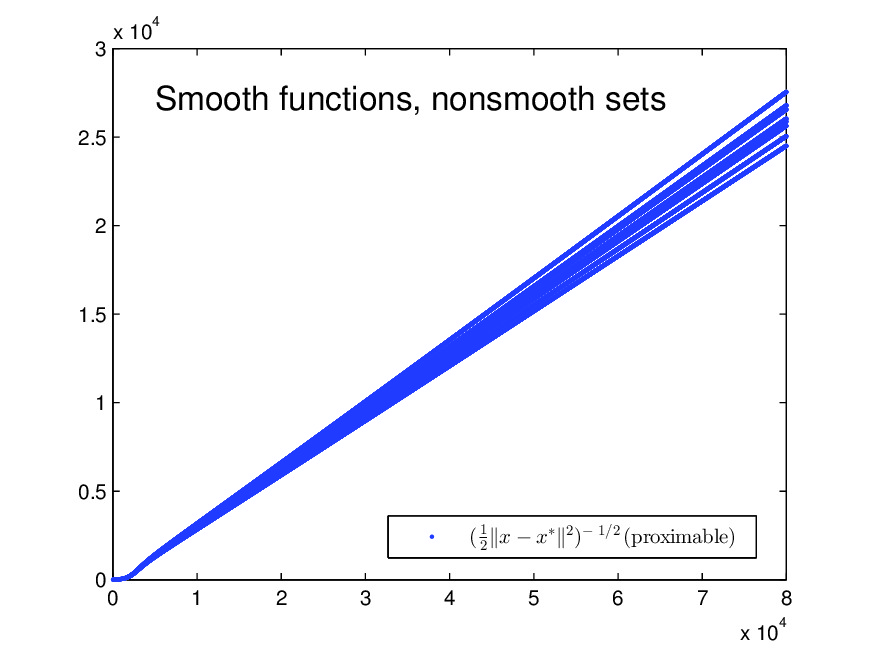}\includegraphics[scale=0.4]{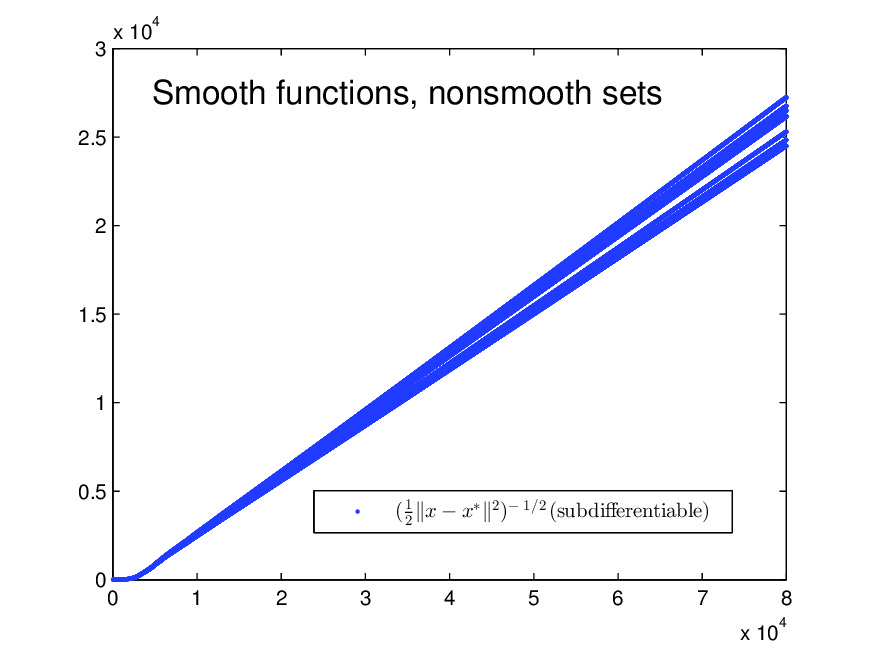}

\caption{\label{fig:exp-3}Plots for Experiment 3. }
\end{figure}

We now elaborate on Figure \ref{fig:exp-4}, which describes a typical
output from the fourth set of experiments. Similar $O(1/k)$ and $O(1/k^{2})$
rates for the convergence of the duality gap and $\frac{1}{2}\|x-x^{*}\|^{2}$
are observed, though our theory so far gives only $O(1/k^{1/3})$
for both quantities, just like what we saw for experiment 3. In the
second figure, the straight line and dashed line correspond to the
case when we treat the nonsmooth functions as proximable and subdifferentiable
functions respectively. Now that the functions $f_{i}(\cdot)$ are
nonsmooth functions, it is now noticeable that if the nonsmooth functions
were treated as proximable functions, the convergence of the duality
gap and $\frac{1}{2}\|x-x^{*}\|^{2}$ to zero is faster than if the
nonsmooth functions were treated as subdifferentiable functions. 

\begin{figure}[h]
\includegraphics[scale=0.4]{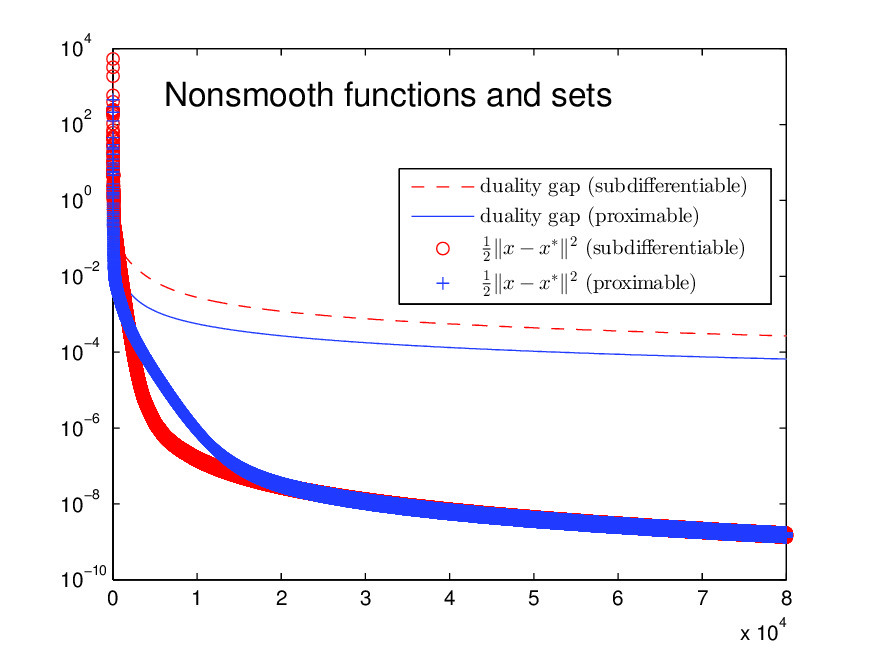}\includegraphics[scale=0.4]{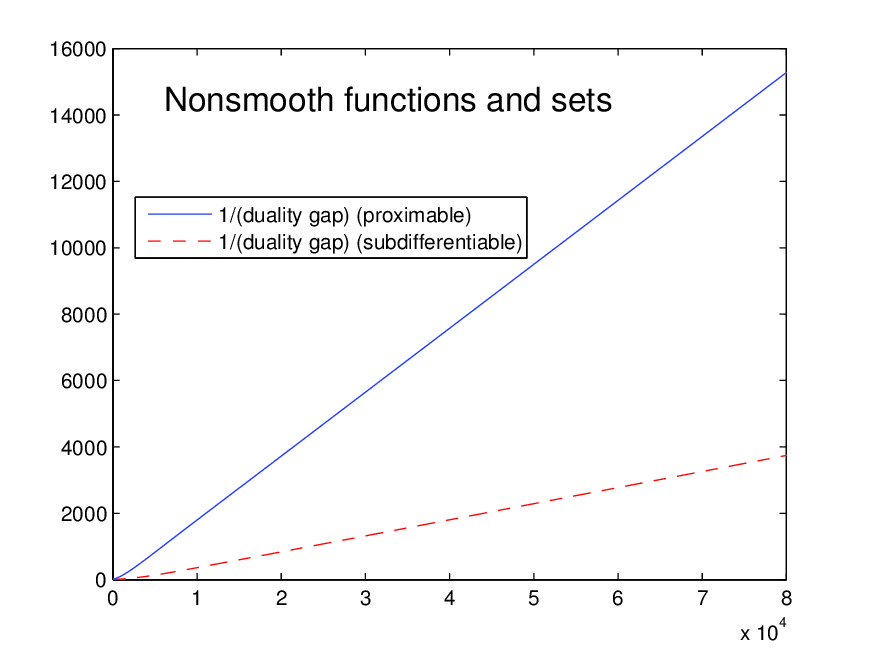}

\includegraphics[scale=0.4]{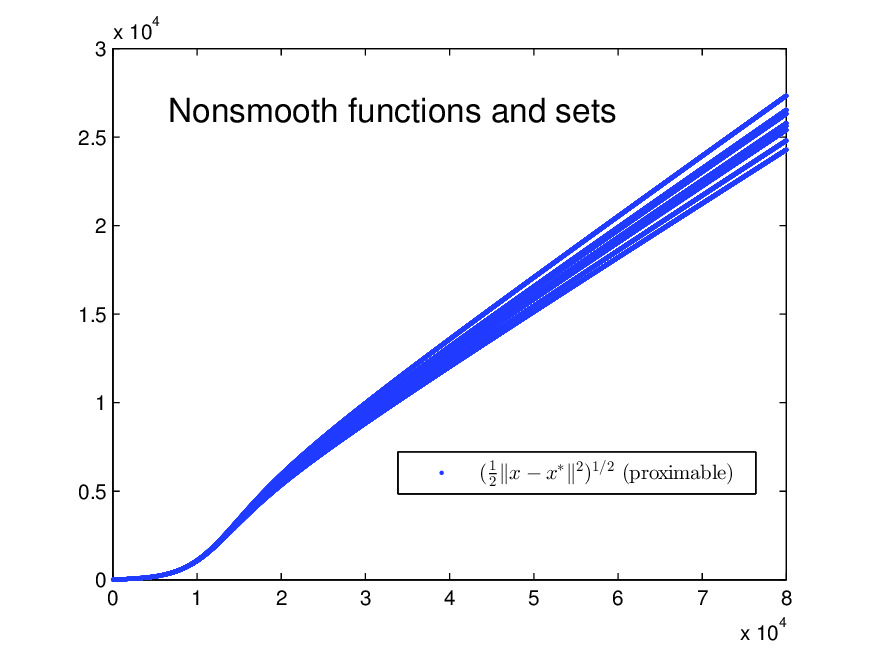}\includegraphics[scale=0.4]{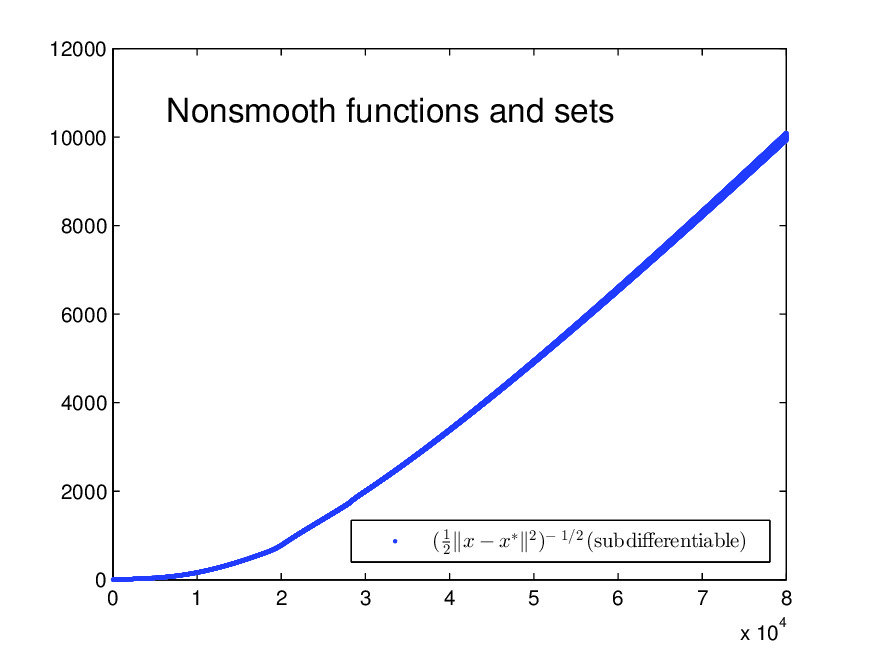}

\caption{\label{fig:exp-4}Plots for Experiment 4. }
\end{figure}
\bibliographystyle{amsalpha}
\bibliography{../refs}

\providecommand{\noop}[1]{}
\providecommand{\bysame}{\leavevmode\hbox to3em{\hrulefill}\thinspace}
\providecommand{\MR}{\relax\ifhmode\unskip\space\fi MR }
\providecommand{\MRhref}[2]{%
  \href{http://www.ams.org/mathscinet-getitem?mr=#1}{#2}
}
\providecommand{\href}[2]{#2}
\begin{thebibliography}{BCRZM03}

\bibitem[BC11]{BauschkeCombettes11}
H.H. Bauschke and P.L. Combettes, \emph{Convex analysis and monotone operator
  theory in {H}ilbert spaces}, Springer, 2011.

\bibitem[BCRZM03]{BCRZ03}
L.M. Bregman, Y.~Censor, S.~Reich, and Y.~Zepkowitz-Malachi, \emph{Finding the
  projection of a point onto the intersection of convex sets via projections
  onto half-spaces}, J. Approx. Theory \textbf{124} (2003), 194--218.

\bibitem[BD85]{BD86}
J.P. Boyle and R.L. Dykstra, \emph{A method for finding projections onto the
  intersection of convex sets in {H}ilbert spaces}, Advances in Order
  Restricted Statistical Inference, Lecture notes in Statistics, Springer, New
  York, 1985, pp.~28--47.

\bibitem[Bec15]{Beck_alt_min_SIOPT_2015}
A.~Beck, \emph{On the convergence of alternating minimization for convex
  programming with applications to iteratively reweighted least squares and
  decomposition schemes}, SIAM J. Optim. \textbf{25} (2015), no.~1, 185--209.

\bibitem[BT13]{Beck_Tetruashvili_2013}
A.~Beck and L.~Tetruashvili, \emph{On the convergence of block coordinate
  descent type methods}, SIAM J. Optim. \textbf{23} (2013), no.~4, 2037--2060.

\bibitem[Com00]{Combettes_SICON_2000}
P.L. Combettes, \emph{Strong convergence of block-iterative outer approximation
  methods for convex optimization}, SIAM J. Control Optim. \textbf{38} (2000),
  no.~2, 538--565.

\bibitem[Dyk83]{Dykstra83}
R.L. Dykstra, \emph{An algorithm for restricted least-squares regression}, J.
  Amer. Statist. Assoc. \textbf{78} (1983), 837--842.

\bibitem[GM89]{Gaffke_Mathar}
N.~Gaffke and R.~Mathar, \emph{A cyclic projection algorithm via duality},
  Metrika \textbf{36} (1989), 29--54.

\bibitem[Han88]{Han88}
S.P. Han, \emph{A successive projection method}, Math. Programming \textbf{40}
  (1988), 1--14.

\bibitem[Hau68]{Haugazeau68}
Y.~Haugazeau, \emph{Sur les in{\'{e}}quations variationnelles et la
  minimisation de fonctionenelles convexes}, Ph.D. thesis, Universit{\'{e}} de
  Paris, 1968.

\bibitem[Pan17]{Pang_Dyk_spl}
C.H.J. Pang, \emph{{D}ykstra splitting and an approximate proximal point
  algorithm for minimizing the sum of convex functions}, arxiv eprint
  1709.09499.

\bibitem[Pan18a]{Pang_Dist_Dyk}
\bysame, \emph{\noop{a}{D}istributed deterministic asynchronous algorithms in
  time-varying graphs through {D}ykstra splitting}, 2018.

\bibitem[Pan18b]{Pang_sub_Dyk}
\bysame, \emph{\noop{b}{S}ubdifferentiable functions and partial data
  communication in a distributed deterministic asynchronous {D}ykstra's
  algorithm}, 2018.

\bibitem[Pan18c]{Pang_rate_D_Dyk}
\bysame, \emph{\noop{c}{L}inear and sublinear convergence rates for a
  subdifferentiable distributed deterministic asynchronous {D}ykstra's
  algorithm}, 2018.

\bibitem[RW98]{RW98}
R.T. Rockafellar and R.J.-B. Wets, \emph{Variational analysis}, Grundlehren der
  mathematischen Wissenschaften, vol. 317, Springer, Berlin, 1998.

\end{thebibliography}

\end{document}